\documentclass[a4paper,11pt]{article}

\usepackage[margin=24mm]{geometry}
\usepackage{graphicx,tikz}
\usetikzlibrary{calc, arrows.meta} 
\usepackage{multirow}
\usepackage{xcolor}

\usepackage[
    backend=biber,
    style=alphabetic,
    maxcitenames=100,
    maxbibnames=100,
    backref=false,
    date=year,
    eprint=false,
    url=false,
    doi=false,
    isbn=false
]{biblatex}
\bibliography{bibliography}
 \AtEveryCitekey{\clearlist{location}}
 \AtEveryBibitem{\clearlist{location}}
\makeatletter
\newcommand\footnoteref[1]{\protected@xdef\@thefnmark{\ref{#1}}\@footnotemark}
\makeatother

\usepackage[algo2e,vlined,ruled,linesnumbered,nofillcomment]{algorithm2e} 
\SetKwIF{If}{ElseIf}{Else}{if}{}{else if}{else}{end if}%
\SetKwFor{While}{while}{}{end while}%
\SetKwFor{For}{for}{}{end for}%
\DontPrintSemicolon

\usepackage{listings}
\lstdefinestyle{snippets}{
    backgroundcolor=\color{bg},   
    commentstyle=\color{ggreen},
    keywordstyle=\color{gred},
    numberstyle=\tiny\color{gray},
    stringstyle=\color{gblue},
    identifierstyle=\color{gblue},
    basicstyle=\ttfamily\footnotesize,
    breakatwhitespace=false,         
    breaklines=true,                 
    captionpos=b,                    
    keepspaces=true,                 
    numbers=left,                    
    numbersep=5pt,                  
    showspaces=false,                
    showstringspaces=false,
    showtabs=false,                  
    tabsize=2,
    xleftmargin=3mm,
    xrightmargin=3mm,
    framexleftmargin=3mm,
    framexrightmargin=3mm,
    framextopmargin=0mm,
    framexbottommargin=0mm,
    frame=tlbr,framesep=5pt,framerule=0pt
}
\usepackage{amsmath,amssymb}
\usepackage{colonequals}

\DeclareMathOperator*{\erule}{\mathrm{er}}

\newcommand\diam{\mathrm{diam}}
\newcommand\distbottleneck{d_b}
\newcommand\R{\mathbf{R}}
\newcommand\Vect{\mathsf{Vect}}
\newcommand\Z{\mathbf{Z}}

\newcommand{\nameshort}{{\sc\mbox{TopoAware}}}
\newcommand{\namelong}{Topologically Aware Constructions for Large and Irregular datasets}

\definecolor{color0}{HTML}{1f77b4}    
\definecolor{color1}{HTML}{ff7f0e}    
\definecolor{color2}{HTML}{2ca02c}    
\definecolor{color3}{HTML}{d62728}    
\definecolor{color4}{HTML}{9467bd}    
\definecolor{color5}{HTML}{8c564b}    

\definecolor{gblue}{HTML}{4285f4}   
\definecolor{ggreen}{HTML}{0f9d58}  
\definecolor{gyellow}{HTML}{f4b400} 
\definecolor{gred}{HTML}{db4437}    

\definecolor{fg}{HTML}{0f9d58}  
\definecolor{bg}{HTML}{cfebde}  

\newcommand\emphasize[1]{{\color{ggreen}\textbf{#1}}}

\usepackage{hyperref} 
\usepackage[noabbrev,capitalize]{cleveref}

\usepackage{amsthm,thmtools}
\newtheorem{theorem}{Theorem}[section]

\theoremstyle{definition}
\newtheorem{definition}[theorem]{Definition}
\newtheorem{remark}[theorem]{Remark}

\crefname{definition}{Definition}{Definitions}

\begin{document}



\title{Stability of 0-dimensional persistent homology in enriched and sparsified point clouds}

\author{
Jānis Lazovskis\footnote{Institute of Clinical and Preventive Medicine, University of Latvia, Riga, Latvia}\ $^\dagger$,
Ran Levi\footnote{Institute of Mathematics, University of Aberdeen, Aberdeen, United Kingdom},
Juliano Morimoto$^\dagger$\footnote{
Programa de Pós-graduação em Ecologia e Conservação, Universidade Federal do Paraná, Curitiba 82590-300, Brazil
}
}

\date{\today}

\maketitle

\begin{abstract}
We give bounds for dimension 0 persistent homology and codimension 1 homology of Vietoris--Rips, alpha, and cubical complex filtrations from finite sets related by enrichment (adding new elements), sparsification (removing elements), and aligning to a grid (uniformly discretizing elements).
For enrichment we use barycentric subdivision, for sparsification we use a minimum separating distance, and for aligning to a grid we take the quotient when dividing each coordinate value by a fixed step size.
We are motivated by applications presenting large and irregular datasets, and the development of persistent homology to better work with them.
In particular, we consider an application to ecology, in which the state of an observed species is inferred through a high-dimensional space with environmental variables as dimensions.
This ``hypervolume'' has geometry (volume, convexity) and topology (connectedness, homology), which are known to be related to the current and potentially future status of the species.
We offer an approach for the analysis of hypervolumes with topological guarantees, complementary to current statistical methods, giving precise bounds between persistence diagrams of Vietoris--Rips and alpha complexes, and a duality identity for cubical complexes.
Implementation of our methods, called {\nameshort}, is made available in C++, Python, and R, building upon the GUDHI library.
\end{abstract}

\noindent
\textbf{keywords:} persistent homology, stability, barycentric subdivision, sparsification, duality, niche, hypervolume, functional ecology

\section{Introduction}

Topological data analysis (TDA) \cite{carlsson2009} has seen a wide variety of mathematical, computational, and applied developments since its inception in the late 1990s and early 2000s \cite{dey_wang_2022,DONUT}.
Most applications of TDA start with a finite set, sourced from some collected or observed data, with a notion of distance between the samples, and the machinery of persistent homology (PH) \cite{robins99} is applied to it.
This process builds a sequence of ever larger topological spaces on this finite set, most often a simplicial or cubical complex, constructed by considering nearby subsets of the input data set, after which the topological features of the individual spaces are tracked as the ``nearness'' parameter increases.
Translating homological observations back into the source of the data is easiest in low dimensions, hence our focus on dimensions 0 and 1. 

In the theme of developing mathematically sound methods for TDA, we consider the application of topology to \emphasize{ecology}, in particular the consideration of the concept of the \emphasize{niche} of a species as a topological space \cite{hutchinson1957, elton1927animal, grinnell1917niche, vandermeer1972niche}.
Recent interest and tools work to better understand this space \cite{blonder2018, qiao2016nichea, escobar2018ecological}, and we contribute to this field with a complementary approach: rather than taking a statistical view of the niche as in  for example \cite{blonder2018} and \cite{qiao2016nichea}, we take a topological view as in \cite{conceiccao2022holey}. In quantitative ecology, a common approach is to model the realized niche from observations of individuals of the species. To achieve this, the environmental variables associated with each observation are used to construct a point cloud corresponding to the combination of environmental variables that species can exist. The collection of points representing this environmental fingerprint is then used alongside statistical tools to create the species' \textit{realised niche hypervolume}, which is a denser (richer) point cloud that includes the true observations and additional points inferred from the statistical models. In this paper, we adopt a different approach: we enrich the  the point cloud by adding in new data points \textit{between} existing observations, rather than \textit{around} a single existing data point as commonly done with statistical approaches.
The examples given are 2-dimensional spaces, as the niche hypervolume is usually analyzed after dimensionality reduction (by e.g. principal component analysis) has been applied, resulting in 2 or 3 principal components. Nevertheless, our methods generalize to any number of dimensions, with a notably limiting computational cost being the alpha complex \cite{Edelsbrunner1995} in higher dimensions.

\subsection{Motivation}
The computation of PH most often runs into technical limits, as the number of simplices involved in computations increases sharply with any increase in the underlying data set or the requested dimension.
We are motivated by 
\begin{itemize}
    \item the desire to effectively execute topological computations in such settings, while maintaining a precise relationship with the original dataset;
    \item the potential to apply our methods to other fields such as ecology, and we include a discussion of this relationship throughout the work.
\end{itemize}
The intended audience for this work is either mathematicians working with large and irregular datasets, which often come from other fields of science, or ecologists wishing to better understand the foundations of computational methods. We present our results in this context.

In ecology, it is paramount to understand how species interact with each other and with their environment. This helps forecast species distribution, extinction risks, and responses to changing climate.
The geometric object defined by the parameters required for survival of the species is its realized niche, and it is then no surprise this object is a cornerstone in ecology \cite{soberon2009niches, chase2009ecological}.
There are three historic and widely accepted definitions of a niche, namely the habitat niche \cite{grinnell1917niche}, functional niche \cite{elton1927animal}, and the Hutchinsonian niche \cite{hutchinson1957}. This paper is motivated by the use and applications of the Hutchinsonian niche hypervolume,  because it provides a quantitative definition and can be used as the conceptual framework for analyzing species niches. 

Given one of these approaches, the \emphasize{fundamental niche} is the complete set of biotic and abiotic factors where a species can in theory persist. However, due to species-species interactions and other ecological and environmental factors, species often occupy a narrower niche than they otherwise could. Thus, this \textit{realization} of the fundamental niche –– known as the \emphasize{realized niche} –– is the subset biotic and abiotic conditions where the species is actually observed. Our methodology is applicable to realized niches, as the fundamental niche of species is rarely, if ever, entirely known. 
Even so, generalizing species niche from limited data is useful not only for averting potential future issues, but for limiting the current harm, for example, when for example working with the UN Red List and inferring new locations of species \cite{cazalis2024}.

\subsection{Related work}
A central part of our methods is reducing the amount of computations by reducing the amount of simplices in a filtration.
Reducing the number of simplices before applying the machinery of persistent homology is well-studied, precisely due to this potential of reducing computational overhead \cite{sheehy2013,Botnan2015,simba}.
Existing approaches often focus on particular constructions, such as the Vietoris--Rips or \v Cech simplicial complexes, and describe bounds using interleavings of persistence modules \cite{chazal2009}.
We narrow our focus on removing 0-simplices, a focus also employed by \cite{Herick2024} by additionally considering a function $f\colon X\to \R$. 
We note that the results of \cite[Theorem 5.5]{Herick2024} correspond with our \cref{thm_mainsparse}, in the case that $f(x)=\epsilon$ is constant.

The bounds we present are in bottleneck distance, between the same constructions on different datasets. 
We work in dimension 0 and 1 homology, often making use of \emphasize{bar-to-bar} maps \cite{Agerberg2025} between persistence diagrams, and work with geometric arguments on the underlying metric space.
A similar perspective, of directly relating elements of the underlying set to elements of the persistence diagram is explored more generally in \cite{gulen2025}.

In \cref{sec_cubicalduality} we construct a point cloud on a grid, computing its homology and the homology of its complement, with a cubical complex filtration.
Applying PH to data arranged in a grid has been developed by the natural context of building sequences of cubical complexes from 2D and 3D images \cite{wagner2012}. 
Data that is almost a grid has been considered through by applications to the rigidity of particles in materials, as ``hyperuniformity" \cite{Salvalaglio2024}.
The process of discretizing data before applying PH has also been studied \cite{dlotko2018}, though more for the context of discretizing continuous data, rather than an already discrete set.

In terms of our main application, the Hutchinsonian niche hypervolume, its intuitive abstract definition and the abundant availability of data \cite{gbif,worldclim} have resulted its broad scientific usage.
Recent developments in mathematical methods for computing geometric properties of the hypervolume have been based on kernel density estimators \cite{blonder2018}. 
This approach is statistical and provides no guaranty against distorting the topological features of the niche hypervolume.
With such an approach, the geometric and topological properties of the hypervolume remain unclear in connection with the input data, warranting our persistence diagram approach, which respects the underlying topology. Previous work by \cite{conceiccao2022holey} have conceptualized the realized niche in a topological approach and more recently, \cite{caron2025evolution} have used this approach to study the evolution of realized niche in woody plants. However, these previous studies were constrained by the computational costs associated with calculating topological features in high dimensional, dense point clouds, thereby limiting their scope.




\subsection{Contribution}
Our work addresses computational issues when working with the PH of point clouds, from a topological perspective.
Addressing computational limits allows our methods to be applied to other fields as well, as the starting point is simply a finite metric space. 

Our main contributions are to present barycentric subdivision as a way to add new elements to a point cloud; to prove the stability of barycentric subdivision, sparsification, and ``gridification'' of a point cloud; and to present duality as a computational tool for computing codimension 1 homology of a point cloud. 
Upon this work we intend to further develop and apply these methods to relevant datasets in the computational ecological community, providing a complementary approach to reach the same goal \cite{blonder2018}.


The underlying mathematical ideas are presented in \cref{sec_topology}, and our main results, along with examples demonstrating sharp bounds, are in \cref{sec_mainresults}.
The first three results (\cref{thm_mainbary,thm_mainsparse,thm_maingrid}) set bounds for the bottleneck distance of degree-0 PH between an input data set and one of the associated structures.
The fourth result (\cref{thm_complement}) uses duality to describe degree-$(N-1)$ homology of a cubical complex on a grid in $\Z^2$, in terms of the degree-0 homology of a coarser cubical complex on the same grid.

\subsection{Outline}
We introduce the necessary topological structures in \cref{sec_topology}, assuming that the reader is familiar with topological concepts like metric spaces, homology, and simplicial complexes (the interested reader is otherwise referred to a source such as \cite{hatcher}).
The main operations that we perform on data sets (barycentric subdivision, sparsification, aligning to a grid) are discussed in detail in \cref{sec_structures}.
In \cref{sec_mainresults} we state and prove our four main results.
An implementation of executing the main operations, as well as an example of executing them is described in \cref{sec_testing}.
We conclude the work with a discussion for applications and continued development in \cref{sec_discussion}.

\section{Topological background}
\label{sec_topology}

Let $(X,d)$ be a finite metric space, in our context most often a subset of $\R^N$ with $d$ the Euclidean metric $d$, and an ordering on the elements.
The distance $d$ is omitted when it is clear from context.

\subsection{Simplicial and cubical complexes}
We will make use of three topological constructions, presented visually in \cref{fig_examples0}.

The \emphasize{Vietoris--Rips} complex $VR_r(X)$ on $X$ is the set $\{\sigma \subseteq X\ :\ d(x,x')\leqslant r\ \forall\ x,x'\in \sigma\}$, and has an advantage of being defined only by computing pairwise distances between points, a computation which is affected much less by the ambient dimension than for other constructions.

The \emphasize{alpha} complex $A_r(X)$ on $X$ is defined \cite{edelsbrunner_harer} using the Voronoi cell $V(x) = \{y\in \R^n\ :\ d(x,y) \leqslant d(x',y)\ \forall\ x'\in X, x\neq x'\}$ at $x$, and the closed ball $B_r(x) = \{y\in \R^n\ :\ d(x,y) \leqslant r\}$ at $x$.
A simplex $\sigma$ exists in $A_r(X)$ if $\bigcap_{x\in \sigma} (B_r(x)\cap V(x)) \neq \emptyset$.
This complex has simplices in dimension at most the ambient dimension (when $X$ is in general position), which usually makes it smaller than the Vietoris--Rips complex in low dimensions.

A \emphasize{cubical} complex is made up of cubes $[0,1]^k$ instead of simplices. 
Given $G\subseteq \mu\Z^n$ for $\mu>0$, a $k$-cube in a cubical complex on $G$ is the image of an  isometry $\{0,\mu\}^k\to G$, and when $G$ is ordered, a $k$-cube is ordered with the highest order value among all its vertices.
This is the lower-star \emphasize{filtration} on $G$ induced by the 0-cubes (an analogous lower-star filtration is induced by ordering the maximal cubes), and the cubical complex of cubes with filtration values at most $r$ is denoted by $\mathcal C_{\mu,r}(G)$.
We call this simply the cubical filtration, and write $\mathcal C_\mu(G)$ for $r\gg 0$.

\begin{figure}[htbp]
    \centering
	\newcommand\pdist{.04} 
\newcommand\gdist{1.5pt} 
\newcommand\framestep{5.1} 
\newcommand\framevstep{-1.5} 
\newcommand\unit{.3} 
\newcommand\labx{0}
\newcommand\laby{.8}
\begin{tikzpicture}
\node[anchor=east] at (\labx,\laby) {$VR_r(X)$};
\foreach \x\y\n in {0.5/.8/a, 1.5/.8/b, 2.2/1.2/c, 2.4/1.2/d, 2/.7/e, 2.1/.5/f}{
  \coordinate (\n) at (\x,\y);
}
\foreach \x\y\z in {c/d/e, b/e/f}{
 \fill[bg] (\x)--(\y)--(\z)--cycle;
}
\foreach \n in {a, b, c, d, e, f}{ \fill[black] (\n) circle (\pdist); }
\foreach \x\y in {c/d, e/f, b/e, c/e, d/e, b/f}{
  \draw (\x)--(\y);
}
\begin{scope}[shift={(0,\framevstep)}]
\node[anchor=east] at (\labx,\laby) {$VR_{r'}(X)$};
\foreach \x\y\n in {0.5/.8/a, 1.5/.8/b, 2.2/1.2/c, 2.4/1.2/d, 2/.7/e, 2.1/.5/f}{
  \coordinate (\n) at (\x,\y);
}
\foreach \x\y\z in {b/e/f, b/c/e, c/d/e, e/f/d, a/b/e}{
 \fill[bg] (\x)--(\y)--(\z)--cycle;
}
\foreach \n in {a, b, c, d, e, f}{ \fill[black] (\n) circle (\pdist); }
\foreach \x\y in {c/d, e/f, b/e, c/e, b/f, c/f, d/f, b/c, b/d, a/b, a/e}{
  \draw (\x)--(\y);
}
\end{scope}
\begin{scope}[shift={(\framestep,0)}]
\node[anchor=east] at (\labx,\laby) {$A_r(X)$};
\foreach \x\y\n in {0.5/.8/a, 1.5/.8/b, 2.2/1.2/c, 2.4/1.2/d, 2/.7/e, 2.1/.5/f}{
  \coordinate (\n) at (\x,\y);
}
\foreach \x\y\z in {c/d/e, b/e/f}{
  \fill[bg] (\x)--(\y)--(\z)--cycle;
 }
\foreach \n in {a, b, c, d, e, f}{ \fill[black] (\n) circle (\pdist); }
\foreach \x\y in {c/d, e/f, b/e, c/e, d/e, b/f}{
  \draw (\x)--(\y);
}
\end{scope}
\begin{scope}[shift={(\framestep,\framevstep)}]
\node[anchor=east] at (\labx,\laby) {$A_{r'}(X)$};
\foreach \x\y\n in {0.5/.8/a, 1.5/.8/b, 2.2/1.2/c, 2.4/1.2/d, 2/.7/e, 2.1/.5/f}{
  \coordinate (\n) at (\x,\y);
}
\foreach \x\y\z in {c/d/e, b/e/f, b/c/e, d/e/f}{
  \fill[bg] (\x)--(\y)--(\z)--cycle;
 }
\foreach \n in {a, b, c, d, e, f}{ \fill[black] (\n) circle (\pdist); }
\foreach \x\y in {c/d, e/f, b/e, c/e, d/e, b/f, d/f, b/c, a/b}{
  \draw (\x)--(\y);
}
\end{scope}
\begin{scope}[shift={(2*\framestep+.4,.3)}]
\node[anchor=east] at (\labx-.4,\laby-.3) {$\mathcal C_{\mu,r}(G)$};
\foreach \x\y\n in {0/2/1, 2/2/2, 2/1/3, 3/1/4, 3/0/5, 3/2/6, 4/2/7, 4/3/8, 5/3/9}{
  \coordinate (\n) at (\x*\unit,\y*\unit);
}
\foreach \n in {1,...,5}{
  \fill[black] (\n) circle (\pdist); 
  \node[anchor=north east,scale=.4] at (\n) {\n};
}
\foreach \n in {6,...,9}{
  \fill[black] (\n) circle (.06); 
  \fill[white] (\n) circle (\pdist); 
  \node[anchor=north east,scale=.4] at (\n) {\n};
}
\foreach \x\y in {2/3, 3/4, 4/5}{
  \draw (\x)--(\y);
}
\end{scope}
\begin{scope}[shift={(2*\framestep+.4,\framevstep+.3)}]
\node[anchor=east] at (\labx-.4,\laby-.3) {$\mathcal C_{\mu,r'}(G)$};
\foreach \x\y\n in {0/2/1, 2/2/2, 2/1/3, 3/1/4, 3/0/5, 3/2/6, 4/2/7, 4/3/8, 5/3/9}{
  \coordinate (\n) at (\x*\unit,\y*\unit);
}
\fill[bg] (2)--(6)--(4)--(3)--cycle;
\foreach \n in {1,...,9}{
  \fill[black] (\n) circle (\pdist); 
  \node[anchor=north east,scale=.4] at (\n) {\n};
}
\foreach \x\y in {2/3, 3/4, 4/5, 4/6, 2/6, 6/7, 7/8, 8/9}{
  \draw (\x)--(\y);
}
\end{scope}
\end{tikzpicture}
	\caption{An overview of the simplicial and cubical complexes used here and their filtrations. From left to right: Vietoris--Rips complexes at radii $r<r'$, alpha complexes at radii $r<r'$, and cubical complexes at filtration values $r=5<r'=9$ (by the ordering on the 0-cubes). On the left, 0-cubes in white are not included in the cubical complex.}
	\label{fig_examples0}
\end{figure}
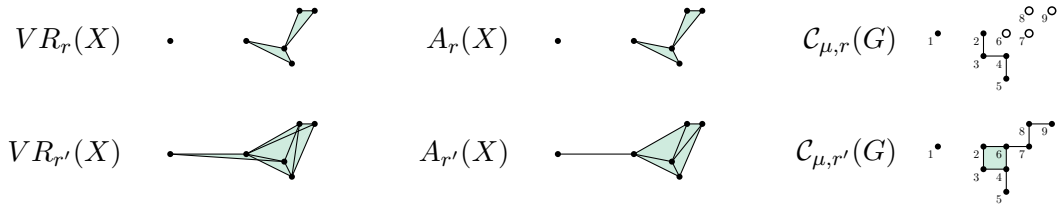 

For all of theses constructions, we must have that if a $k$-simplex or $k$-cube is in the respective complex, then every face $(k-1)$-simplex or $(k-1)$-cube must also exist in the same complex.
Then, given radii $r<r'$, we get inclusions $VR_r(X) \subseteq VR_{r'}(X)$ and $A_r(X)\subseteq A_{r'}(X)$, so we often refer the the Vietoris--Rips and alpha \emphasize{filtrations} of a metric space $(X,d)$.

\subsection{Persistent homology}

Let $K=K(X)$ be a simplicial or cubical complex on a finite metric space $(X,d)$, with a filtration
\begin{equation}\label{eqn_filtration}    
\mathcal F\colon \emptyset = K_0 \subseteq K_{r'} \subseteq \cdots \subseteq K_{r''} = K, 
\end{equation}
where $r\in \R_{\geqslant 0}$ is the increasing filtration parameter.
We consider persistent homology as a functor $PH_*\colon \R\to \Vect$ from the category of real numbers and the partial order $\leqslant$, to the category of vector spaces and linear maps.
Similarly, we consider the persistence diagram $D_*$ as the multiset of intervals in $\R_{\geqslant 0}$ indicating the nonzero parts of the indecomposable complexes in the image of the functor $PH_*$, where the asterisk $*$ denotes a fixed dimension.
When a finite set $X$ is fixed, we write $PH_*(X)$ for the functor $r\mapsto H_*(K_r)$ and $(r\leqslant s)\mapsto (H_*(K_r) \to H_*(K_s)$, and $D_X$ for the image of this functor.
For $K=\mathcal C$ the cubical complex, we are concerned with cases when the homology $H_*$ always has a fixed filtration parameter.

\begin{definition}\cite{edelsbrunner_harer}
\label{def_points2bars}
Let $X\subseteq \R^N$ be finite and ordered, and $D_X$ the persistence diagram of $PH_0(X)$ of the complex $K(X)$ with filtration $\mathcal F$. Write $\erule_X\colon X \xrightarrow{\ \simeq\ } D_X$ for the bijection defined by the \emphasize{elder rule}. 
\end{definition}

The action of the elder rule may be considered constructively as follows.
Consider $K$ with filtration $\mathcal F$, for which there is $r,\epsilon>0$ such that two distinct connected components at $K_r$ are within the same connected component at $K_{r+\epsilon}$.
This means there is an interval $[0,r')\in D_K$, with $r'\in(r,r+\epsilon)$.
In such a case, one of the two components at $K_r$ must contain an element $x_0$ ordered strictly before all the elements of the other component.
This other component (the one not containing this $x_0$) is then, by the elder rule, associated with the interval $[0,r')$, as the first component is ``older" and so will be associated with a longer interval. 

\begin{figure}[htbp]
    \centering
	\newcommand\unit{.5} 
\newcommand\overlap{.25} 
\begin{tikzpicture}[
  reg/.style={line width=1pt,draw=black!50},
  lab/.style={inner sep=7pt},
  alab/.style={fill=white,fill opacity=.8,text opacity=1},
  brr/.style={line width=1pt,dashed,{Rectangle[right,length=1pt,width=10pt]}-{Rectangle[left,length=1pt,width=10pt]}},
  brl/.style={line width=1pt,dashed,{Rectangle[left,length=1pt,width=10pt]}-{Rectangle[right,length=1pt,width=10pt]}}
]
\foreach \x\y\n in {
106/138/6, 119/167/7, 84/151/8, 131/139/9, 135/173/10, 158/186/11, 186/233/12, 95/181/15, 140/213/16, 169/259/17, 203/292/18, 227/280/19, 772/171/20, 197/273/21, 255/313/24, 285/290/25, 254/286/26, 225/313/27, 175/285/28, 283/323/29, 760/148/30, 408/285/32, 404/262/33, 341/279/35, 376/294/36, 347/302/37, 320/309/38, 465/246/40, 431/289/42, 435/247/43, 370/266/45, 443/262/46, 483/269/47, 463/280/49, 500/205/50, 490/231/51, 521/218/52, 505/257/54,  380/79/56, 400/90/57, 371/57/58, 425/75/61, 443/111/62, 485/116/63, 520/109/64, 454/85/65, 548/89/67, 576/105/68, 555/111/69, 641/118/72, 642/88/73, 676/85/74, 713/100/75, 690/112/77, 664/110/78, 732/119/79, 709/133/83, 675/137/84, 732/151/85, 309/327/86, 122/197/87, 760/194/89, 801/213/90, 794/193/95, 734/176/99, 780/223/100,  315/283/109, 168/212/110, 198/250/111, 163/243/112, 403/60/115, 481/93/117, 580/77/118, 618/85/119, 510/84/121, 606/109/122}{
  \coordinate (\n) at (.01*\x,.01*\y);
  \fill[black] (\n) circle[radius=1.5pt];
}
\node[anchor=east,inner sep=7pt] at (8) {$x_0$};
\node[anchor=east,inner sep=7pt] at (58) {$x_1$};
\draw (8) circle [radius=3pt];
\draw (58) circle [radius=3pt];
\node (x0p) at (6.9,2.8) {$x_0'$};
\node (x1p) at (6.9,2.1) {$x_1'$};
\draw[-{Straight Barb[width=4pt,length=2pt]},shorten >=4pt] (x0p.west) to [out=180,in=330] (50);
\draw[-{Straight Barb[width=4pt,length=2pt]},shorten >=4pt] (x1p.west) to [out=180,in=30] (63);
\draw[ggreen,|-|,transform canvas={shift={(-6pt,1.2pt)}}] (50) to node[ggreen,auto,swap,pos=.8] {$d(x_0',x_1')$} (63);
\begin{scope}[shift={(10,.3)}]
\newcounter{mycounter}
\setcounter{mycounter}{0}
\foreach \x in {12.3693, 13.3417, 17, 17.088, 19.9249, 20.8087, 21.095, 21.1896, 21.8403, 22.8254, 22.8254, 23.0217, 23.0868, 23.2594, 23.3452, 23.3452, 23.3452, 23.7697, 23.7697, 24.0832, 24.1868, 24.3516, 24.6982, 25.0599, 25.0599, 25.0799, 25.0799, 25.9422, 25.9422, 25.9422, 26.0768, 26.3059, 26.3059, 26.4197, 26.4764, 26.6271, 26.8328, 26.8328, 26.8701, 26.9072, 26.9258, 26.9258, 27.0185, 27.2029, 27.2947, 27.6586, 27.6586, 27.7308, 27.8568, 27.8568, 28.1603, 28.1603, 28.2312, 28.2843, 28.6356, 29.1548, 29.1548, 29.1548, 29.2062, 29.5466, 29.7321, 30, 30.0167, 30.0832, 30.2655, 30.3645, 30.6757, 30.8058, 31.1127, 31.3847, 31.9531, 33.2415, 34.4093, 90.2552, 180}{
  \draw[line width=.2pt] (0,.04*\themycounter)--(.02*\x,.04*\themycounter);
  \stepcounter{mycounter}
}
\foreach \x\l\c in {0/0/black, 1.81/{d(x_0',x_1')}/ggreen, 3.6/{\infty}/black}{
  \draw[\c] (\x,3.1) --++ (90:.2) node[\c,above] {$\l$};
}
\end{scope}
\end{tikzpicture}
	\caption{An ordered set $X\subseteq \R^2$ (left) and the persistence diagram $D_X$ of its Vietoris--Rips filtration in degree 0 (right). The first two elements $x_0,x_1\in X$ and the elements $x_0',x_1'$ whose distance defines the second longest element of $D_X$ are emphasized.}
	\label{fig_barcodekilled}
\end{figure}
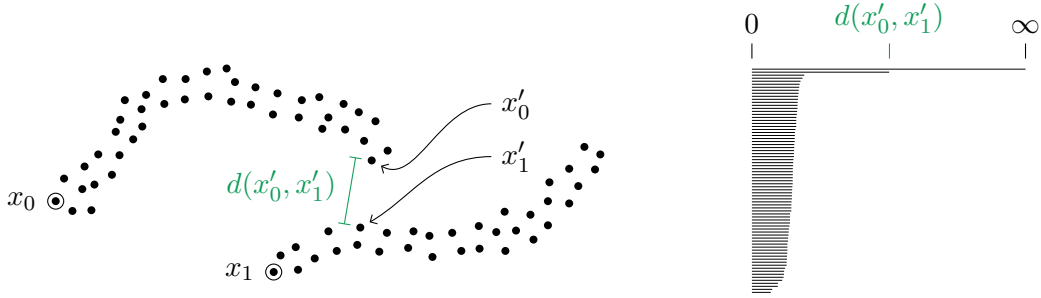 

The elder rule is defined in a similar manner for higher dimensional diagrams, but the bijection between the underlying set and the multiset of intervals only exists in dimension 0.
We now introduce some additional terminology to make precise this relationship.

\begin{definition}
\label{def_killer}
Let $X$ be ordered. Take $x,y\in X$ with $x$ ordered before $y$, and $r>0$ satisfying 
\begin{itemize}
    \item $x$ and $y$ are in different connected components of $K_r$, and 
    \item $x$ and $y$ are the earliest ordered elements in their respective components.
\end{itemize}
Let $x',y'\in X$ be in the connected components at $r$ of $x,y$, respectively, such that for every $\epsilon>0$ for which $x,y$ are in the same connected component of $K_{r+\epsilon}$, the 1-simplex with boundary $\{x',y'\}$ is also in $K_{r+\epsilon}$.
Then we say that $\erule_X(y)\in D_X$ is \emphasize{killed} by $\erule_X(x)$ and by $(x',y')$. 
\end{definition}

This definition makes precise the elements at the shortest distance between two disjoint subsets of $X$.
For example, in \cref{fig_barcodekilled}, $\erule_X(x_1)$ is killed by $\erule_X(x_0)$ at $d(x_0',x_1')$, and killed by $(x_0',x_1')$.
Note that $x',y'$ in \cref{def_killer} are not necessarily unique, but it is always true that $\erule_X(y) = [0,d(x',y')/2)$.

\begin{definition}
The \emphasize{bottleneck distance} between two persistence diagrams $D,D'$ is 
\begin{equation}
\distbottleneck(D,D') \colonequals \inf_{\varphi \colon D \to D'} \left[ \sup_{c\in D} d(c,\varphi(c))\right],
\end{equation}
where $\varphi$ is a bijection and both persistence diagrams are enriched with points countably many points $(r,r)$ for every $r\in \R_{\geqslant 0}$.
An element $c=[c_0,c_1)$ of a persistence diagram may be considered as a pair $(c_0,c_1)\in \R^2$, so the function $d\colon \R^2\times \R^2 \to \R$ is just Euclidean distance in $\R^2$.
\end{definition}


\subsection{Modifying point clouds}
\label{sec_structures}

While we endeavour to construct technically sound methods, we allow ourselves several conversational simplifications: we use \emphasize{point cloud} to mean a finite metric space and \emphasize{enrichment} of a point cloud to mean the union of the point cloud with another point cloud.

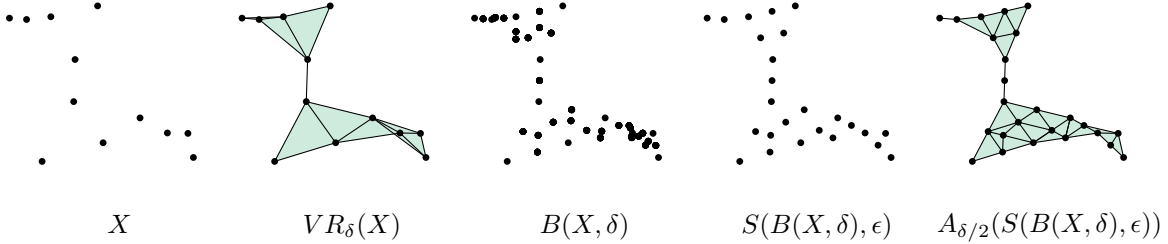
\begin{figure}[htbp]
    \centering
	\newcommand\pdist{.05} 
\newcommand\framestep{3.4} 
\begin{tikzpicture}[scale=.9]
\foreach \x\y\n in {0.796/2.49/a, 2.88/0.417/b, 2.8/0.772/c, 1.48/2.65/d, 0.668/0.359/e, 2.5/0.778/f, 1.13/1.24/g, 0.191/2.47/h, 0.441/2.45/i, 1.56/0.634/j, 1.15/1.86/k, 2.1/1/l}{
  \coordinate (\n) at (\x,\y);
  \fill[black] (\n) circle (\pdist);
}
\begin{scope}[shift={(\framestep,0)}]
\foreach \x\y\n in {0.796/2.49/a, 2.88/0.417/b, 2.8/0.772/c, 1.48/2.65/d, 0.668/0.359/e, 2.5/0.778/f, 1.13/1.24/g, 0.191/2.47/h, 0.441/2.45/i, 1.56/0.634/j, 1.15/1.86/k, 2.1/1/l}{
  \coordinate (\n) at (\x,\y);
}
\foreach \a\b\c in {a/d/k, a/h/i, a/i/k, b/c/f, b/c/l, b/f/l, c/f/l, e/g/j, f/j/l, g/j/l}{\fill[bg] (\a)--(\b)--(\c)--cycle;}
\foreach \a\b in {a/d, a/h, a/i, a/k, b/c, b/f, b/l, c/f, c/l, d/k, e/g, e/j, f/j, f/l, g/j, g/k, g/l, h/i, i/k, j/l}{\draw (\a)--(\b);}
\foreach \n in {a, b, c, d, e, f, g, h, i, j, k, l}{
  \fill[black] (\n) circle (\pdist);
}
\end{scope}
\begin{scope}[shift={(2*\framestep,0)}]
\foreach \x\y in {0.796/2.49, 2.88/0.417, 2.8/0.772, 1.48/2.65, 0.668/0.359, 2.5/0.778, 1.13/1.24, 0.191/2.47, 0.441/2.45, 1.56/0.634, 1.15/1.86, 2.1/1, 1.14/2.57, 0.494/2.48, 0.619/2.47, 0.973/2.18, 2.84/0.595, 2.69/0.598, 2.49/0.709, 2.65/0.775, 2.45/0.886, 1.31/2.25, 0.899/0.799, 1.11/0.496, 2.03/0.706, 2.3/0.889, 1.34/0.937, 1.14/1.55, 1.61/1.12, 0.316/2.46, 0.795/2.16, 1.83/0.817, 1.14/2.57, 0.494/2.48, 0.619/2.47, 0.973/2.18, 2.84/0.595, 2.69/0.598, 2.49/0.709, 2.65/0.775, 2.45/0.886, 1.31/2.25, 0.899/0.799, 1.11/0.496, 2.03/0.706, 2.3/0.889, 1.34/0.937, 1.14/1.55, 1.61/1.12, 0.316/2.46, 0.795/2.16, 1.83/0.817, 1.14/2.33, 0.476/2.47, 0.796/2.27, 2.73/0.656, 2.59/0.73, 2.49/0.732, 2.47/0.85, 1.12/0.744, 2.05/0.804, 1.6/0.958, 1.14/2.57, 0.494/2.48, 0.619/2.47, 0.973/2.18, 2.84/0.595, 2.69/0.598, 2.49/0.709, 2.65/0.775, 2.45/0.886, 1.31/2.25, 0.899/0.799, 1.11/0.496, 2.03/0.706, 2.3/0.889, 1.34/0.937, 1.14/1.55, 1.61/1.12, 0.316/2.46, 0.795/2.16, 1.83/0.817, 1.14/2.33, 0.476/2.47, 0.796/2.27, 2.73/0.656, 2.59/0.73, 2.49/0.732, 2.47/0.85, 1.12/0.744, 2.05/0.804, 1.6/0.958, 1.14/2.57, 0.494/2.48, 0.619/2.47, 0.973/2.18, 2.84/0.595, 2.69/0.598, 2.49/0.709, 2.65/0.775, 2.45/0.886, 1.31/2.25, 0.899/0.799, 1.11/0.496, 2.03/0.706, 2.3/0.889, 1.34/0.937, 1.14/1.55, 1.61/1.12, 0.316/2.46, 0.795/2.16, 1.83/0.817, 1.14/2.33, 0.476/2.47, 0.796/2.27, 2.73/0.656, 2.59/0.73, 2.49/0.732, 2.47/0.85, 1.12/0.744, 2.05/0.804, 1.6/0.958, 1.14/2.57, 0.494/2.48, 0.619/2.47, 0.973/2.18, 2.84/0.595, 2.69/0.598, 2.49/0.709, 2.65/0.775, 2.45/0.886, 1.31/2.25, 0.899/0.799, 1.11/0.496, 2.03/0.706, 2.3/0.889, 1.34/0.937, 1.14/1.55, 1.61/1.12, 0.316/2.46, 0.795/2.16, 1.83/0.817, 1.14/2.33, 0.476/2.47, 0.796/2.27, 2.73/0.656, 2.59/0.73, 2.49/0.732, 2.47/0.85, 1.12/0.744, 2.05/0.804, 1.6/0.958, 1.14/2.57, 0.494/2.48, 0.619/2.47, 0.973/2.18, 2.84/0.595, 2.69/0.598, 2.49/0.709, 2.65/0.775, 2.45/0.886, 1.31/2.25, 0.899/0.799, 1.11/0.496, 2.03/0.706, 2.3/0.889, 1.34/0.937, 1.14/1.55, 1.61/1.12, 0.316/2.46, 0.795/2.16, 1.83/0.817, 1.14/2.33, 0.476/2.47, 0.796/2.27, 2.73/0.656, 2.59/0.73, 2.49/0.732, 2.47/0.85, 1.12/0.744, 2.05/0.804, 1.6/0.958, 1.14/2.57, 0.494/2.48, 0.619/2.47, 0.973/2.18, 2.84/0.595, 2.69/0.598, 2.49/0.709, 2.65/0.775, 2.45/0.886, 1.31/2.25, 0.899/0.799, 1.11/0.496, 2.03/0.706, 2.3/0.889, 1.34/0.937, 1.14/1.55, 1.61/1.12, 0.316/2.46, 0.795/2.16, 1.83/0.817, 1.14/2.33, 0.476/2.47, 0.796/2.27, 2.73/0.656, 2.59/0.73, 2.49/0.732, 2.47/0.85, 1.12/0.744, 2.05/0.804, 1.6/0.958, 1.14/2.57, 0.494/2.48, 0.619/2.47, 0.973/2.18, 2.84/0.595, 2.69/0.598, 2.49/0.709, 2.65/0.775, 2.45/0.886, 1.31/2.25, 0.899/0.799, 1.11/0.496, 2.03/0.706, 2.3/0.889, 1.34/0.937, 1.14/1.55, 1.61/1.12, 0.316/2.46, 0.795/2.16, 1.83/0.817, 1.14/2.33, 0.476/2.47, 0.796/2.27, 2.73/0.656, 2.59/0.73, 2.49/0.732, 2.47/0.85, 1.12/0.744, 2.05/0.804, 1.6/0.958}{\fill[black] (\x,\y) circle (\pdist);}
\end{scope}
\begin{scope}[shift={(3*\framestep,0)}]
\foreach \x\y in {0.796/2.49, 2.88/0.417, 2.8/0.772, 1.48/2.65, 0.668/0.359, 2.5/0.778, 1.13/1.24, 0.191/2.47, 0.441/2.45, 1.56/0.634, 1.15/1.86, 2.1/1, 1.14/2.57, 0.973/2.18, 2.69/0.598, 1.31/2.25, 0.899/0.799, 1.11/0.496, 2.03/0.706, 2.3/0.889, 1.34/0.937, 1.14/1.55, 1.61/1.12, 1.83/0.817, 1.12/0.744}{\fill[black] (\x,\y) circle (\pdist);}
\end{scope}
\begin{scope}[shift={(4*\framestep,0)}]
\foreach \x\y\n in {0.796/2.49/1, 2.88/0.417/2, 2.8/0.772/3, 1.48/2.65/4, 0.668/0.359/5, 2.5/0.778/6, 1.13/1.24/7, 0.191/2.47/8, 0.441/2.45/9, 1.56/0.634/10, 1.15/1.86/11, 2.1/1/12, 1.14/2.57/13, 0.973/2.18/14, 2.69/0.598/15, 1.31/2.25/16, 0.899/0.799/17, 1.11/0.496/18, 2.03/0.706/19, 2.3/0.889/20, 1.34/0.937/21, 1.14/1.55/22, 1.61/1.12/23, 1.83/0.817/24, 1.12/0.744/25}{
  \coordinate (\n) at (\x,\y);
}
\fill[bg] (9) -- (1) -- (13) -- (4) -- (16) -- (11) -- (14) -- cycle;
\fill[bg] (7) -- (23) -- (12) -- (20) -- (3) -- (2) -- (15) -- (6) -- (19) -- (10) -- (18) -- (5) -- (17) -- cycle;
\foreach \a\b in {8/9, 9/1, 1/13, 13/4, 1/14, 13/14, 13/16, 14/16, 14/11, 16/11, 16/4,  11/22, 22/7, 7/21, 17/25, 25/21, 5/17,  5/18, 17/18, 25/18, 24/10, 18/10, 10/24, 23/24, 24/19, 24/12, 12/19, 12/20, 20/6, 6/15, 15/3, 15/2, 12/19, 23/12, 21/10, 21/24, 21/23, 7/23, 20/19, 6/3, 25/10, 17/21, 17/7, 3/2, 10/19, 19/6, 20/3, 9/14}{
  \draw(\a)--(\b);
}
\foreach \n in {1,...,25}{
  \fill[black] (\n) circle (\pdist);
}
\end{scope}
\begin{scope}[shift={(1.8,-.6)}]
\node (l1) at (0*\framestep,0) {\small\vphantom{Ap}$X$};
\node (l2) at (1*\framestep,0) {\small\vphantom{Ap}$VR_\delta(X)$};
\node (l3) at (2*\framestep,0) {\small\vphantom{Ap}$B(X,\delta)$};
\node (l4) at (3*\framestep,0) {\small\vphantom{Ap}$S(B(X,\delta),\epsilon)$};
\node (l5) at (4*\framestep,0) {\small\vphantom{Ap}$A_{\delta/2}(S(B(X,\delta),\epsilon))$};
\end{scope}
\end{tikzpicture}
	\caption{An overview of point cloud modifications. From left to right: A set $X\subseteq \R^2$, its Vietoris--Rips complex, its barycentric subdivision, the sparsification of its barycentric subdivision, and the alpha complex of the sparsification of its barycentric subdivision. The values chosen for this example are $\delta = .3\cdot \diam(X)$ and $\varepsilon=.06\cdot \diam(X)$.}
	\label{fig_barysparseex}
\end{figure} 

Given a simplicial complex $S$ on a set $X$, the \emphasize{barycentric subdivision} of $S$ is a simplicial complex constructed by decomposing each $n$-simplex in $S$ into a collection of smaller simplices (a standard method fully described in, for example, \cite{hatcher}).
For our purposes, we are interested in the barycentric subdivision of only 1-simplices and 2-simplices of $S$, but our methods extend to any dimension.
Given a fixed length $\delta>0$, we write $B(X,\delta)$ for the 0-simplices of the barycentric subdivision of only 1-simplices and 2-simplices of the Vietoris--Rips simplicial complex $VR_\delta(X)$.
As a set,
\begin{equation}
\label{eqn_barycentricsubdivision}
B(X,\delta) \colonequals X \cup \bigcup_{\sigma=\{x_0,x_1\} \in VR_\delta(X)_1} \left\{ \frac{x_0+x_1}2 \right\} \cup \bigcup_{\sigma=\{x_0,x_1,x_2\} \in  VR_\delta(X)_2} \left\{ \frac{x_0+x_1+x_2}3 \right\},
\end{equation}
where $VR_\delta(X)_k$ is the $k$-skeleton, or the collection of $k$-simplices of $VR_\delta(X)$. 
The elements in $B(X,\delta)$ are ordered first as in $X$, then the midpoints of the 1-simplices, then the centroids of the 2-simplices, and order within each collection induced lexicographically from the order of $X$.

The notion of \emphasize{sparsification} that we use follows the approach implemented in the GUDHI library \cite{gudhi:SubSampling}, though this is not the only method to subsample a point cloud.
Sparsification is also sometimes known as computing \emphasize{landmarks} for a point cloud.

\begin{definition}
\label{def_sparsification}
Given a fixed length $\varepsilon>0$, we write $S(X,\varepsilon)\subseteq X$ for the subset of an ordered set $X$ for which $d(x,x') > \varepsilon$, for every $x\neq x'\in S(X,\varepsilon)$. This set is unique if constructed by proceeding through $X$ by its order, at $x_i$ removing all $x_{j>i}$ with $d(x_i,x_j)\leqslant \varepsilon$.
\end{definition}

\begin{remark}
Note that computing landmarks in this way with two different orders on $X$ will result in point clouds within a Hausdorff distance of $\varepsilon$ of each other.
This follows as every $x\in X$ must have at least one landmark in its $\varepsilon$-ball $B_\varepsilon(X)$, for any ordering of $X$.
\end{remark}

We refer to the image of the natural inclusion of $\Z^N$ into $\R^N$, composed with a translation $z\in \R^N$ and scaling factor $\mu\in \R$ as a $(\mu,z)$-grid, or simply a  \emphasize{grid} when the parameters $\mu,z$ are clear from context or not relevant.
As a set,
\begin{equation}
    \label{eqn:grid}
    G(\mu,z) \colonequals \left\{(n_1\mu,\dots,n_N\mu) + z\in \R^N\ :\ n_i\in \Z\ \forall\ i\right\} \equiv \Z^N,
\end{equation}
and we write $\mu\Z^N$ for brevity, when the $z$ is implied or not relevant.
For our work, the $z$ is indeed not relevant, and may be considered as fixed at $0\in \R^N$ throughout, so we say ``$\mu$-grid'' to refer to $G(\mu,z)$.

\begin{figure}[htbp]
    \centering
	\newcommand\pdist{.04} 
\newcommand\gdist{1.5pt} 
\newcommand\framestep{4} 
\newcommand\unit{.35} 
\begin{tikzpicture}[
  gridline/.style={bg,line width=.8pt},
  shapeline/.style={fg,line width=1pt}
]
\foreach \x in {-1,...,7}{\draw[gridline] (\x*\unit,-.5*\unit)--(\x*\unit,7.5*\unit);}
\foreach \y in {0,...,7}{\draw[gridline] (-1.5*\unit,\y*\unit)--(7.5*\unit,\y*\unit);}
\foreach \x\y in {0/5, 1/5, 2/5, 3/5, 2/4, 2/3, 2/2, 3/2, 4/2, 5/2, 1/1, 2/1, 3/1, 4/1, 5/1, 6/1}{
  \fill[fg] (\x*\unit,\y*\unit) circle[radius=\gdist];
}
\foreach \x\y in {0.796/2.49, 2.88/0.417, 2.8/0.772, 1.48/2.65, 0.668/0.359, 2.5/0.778, 1.13/1.24, 0.191/2.47, 0.441/2.45, 1.56/0.634, 1.15/1.86, 2.1/1, 1.14/2.57, 0.973/2.18, 2.69/0.598, 1.31/2.25, 0.899/0.799, 1.11/0.496, 2.03/0.706, 2.3/0.889, 1.34/0.937, 1.14/1.55, 1.61/1.12, 1.83/0.817, 1.12/0.744}{
  \fill[black] (\x*.7+.1,\y*.7+.2) circle (\pdist);
}
\node at (3*\unit,-2*\unit) {\vphantom{Ap}$X$ \hspace{.3cm} {\color{fg}$G(X,\mu)$}};
\begin{scope}[shift={(\framestep,0)}]
\foreach \x in {-1,...,7}{\draw[gridline] (\x*\unit,-.5*\unit)--(\x*\unit,7.5*\unit);}
\foreach \y in {0,...,7}{\draw[gridline] (-1.5*\unit,\y*\unit)--(7.5*\unit,\y*\unit);}
\foreach \x\y in {0/5, 1/5, 2/5, 3/5, 2/4, 2/3, 2/2, 3/2, 4/2, 5/2, 1/1, 2/1, 3/1, 4/1, 5/1, 6/1}{
  \coordinate (\x\y) at (\x*\unit,\y*\unit);
}
\fill[bg] (21) -- (51) -- (52) -- (22) -- cycle;
\foreach \a\b in {11/21, 21/31, 31/41, 41/51, 51/61, 22/21, 22/32, 32/31, 32/42, 42/41, 42/52, 52/51, 24/22, 05/15, 15/25, 25/24, 25/35}{
  \draw[shapeline] (\a) -- (\b);
}
\foreach \x\y in {0/5, 1/5, 2/5, 3/5, 2/4, 2/3, 2/2, 3/2, 4/2, 5/2, 1/1, 2/1, 3/1, 4/1, 5/1, 6/1}{
  \fill[fg] (\x*\unit,\y*\unit) circle[radius=\gdist];
}
\node at (3*\unit,-2*\unit) {\vphantom{Ap}$C_\mu(G(X,\mu))$};
\end{scope}
\begin{scope}[shift={(2*\framestep,0)}]
\foreach \x in {-1,...,7}{\draw[gridline] (\x*\unit,-.5*\unit)--(\x*\unit,7.5*\unit);}
\foreach \y in {0,...,7}{\draw[gridline] (-1.5*\unit,\y*\unit)--(7.5*\unit,\y*\unit);}
\foreach \x\y in {
-1/7, 0/7, 1/7, 2/7, 3/7, 4/7, 5/7, 6/7, 7/7,
-1/6, 0/6, 1/6, 2/6, 3/6, 4/6, 5/6, 6/6, 7/6,
-1/5, 4/5, 5/5, 6/5, 7/5,
-1/4, 0/4, 1/4, 3/4, 4/4, 5/4, 6/4, 7/4,
-1/3, 0/3, 1/3, 3/3, 4/3, 5/3, 6/3, 7/3,
-1/2, 0/2, 1/2, 6/2, 7/2,
-1/1, 0/1, 7/1,
-1/0, 0/0, 1/0, 2/0, 3/0, 4/0, 5/0, 6/0, 7/0
}{
 \fill[ggreen] (\x*\unit,\y*\unit) circle[radius=\gdist];
}
\node at (3*\unit,-2*\unit) {\vphantom{Ap}$\mu\Z^2\setminus G(X,\mu)$};
\end{scope}
\begin{scope}[shift={(3*\framestep,0)}]
\foreach \x in {-1,...,7}{\draw[gridline] (\x*\unit,-.5*\unit)--(\x*\unit,7.5*\unit);}
\foreach \y in {0,...,7}{\draw[gridline] (-1.5*\unit,\y*\unit)--(7.5*\unit,\y*\unit);}
\foreach \x in {-1,...,7}{
  \foreach \y in {0,...,7}{
    \coordinate (\x\y) at (\x*\unit,\y*\unit);
  }
}
\fill[bg] (-17) -- (77) -- (72) -- (62) -- (63) -- (33) -- (34) -- (44) -- (46) -- (-16) -- cycle; 
\fill[bg] (-10) -- (00) -- (02) -- (12) -- (14) -- (-14) -- cycle;
\foreach \a\b in {-10/70, -11/01, -12/12, 62/72, -13/13, 33/73, -14/14, 34/74, 45/75, -16/76, -17/77, -10/-17, 00/04, 06/07, 12/14, 16/17, 26/27, 33/34, 36/37, 43/47, 53/57, 62/67, 70/77}{
  \draw[shapeline] (\a) -- (\b);
}
\foreach \x\y in {
-1/7, 0/7, 1/7, 2/7, 3/7, 4/7, 5/7, 6/7, 7/7,
-1/6, 0/6, 1/6, 2/6, 3/6, 4/6, 5/6, 6/6, 7/6,
-1/5, 4/5, 5/5, 6/5, 7/5,
-1/4, 0/4, 1/4, 3/4, 4/4, 5/4, 6/4, 7/4,
-1/3, 0/3, 1/3, 3/3, 4/3, 5/3, 6/3, 7/3,
-1/2, 0/2, 1/2, 6/2, 7/2,
-1/1, 0/1, 7/1,
-1/0, 0/0, 1/0, 2/0, 3/0, 4/0, 5/0, 6/0, 7/0
}{
 \fill[ggreen] (\x*\unit,\y*\unit) circle[radius=\gdist];
}
\node at (3*\unit,-2*\unit) {\vphantom{Ap}$\mathcal C_\mu(\mu\Z^2\setminus G(X,\mu))$};
\end{scope}
\end{tikzpicture}
	\caption{An overview of constructions on grids. From left to right: A set $X\subseteq \R^2$ (taken from \cref{fig_barysparseex}, second right) overlaid by the grid $G(X,\mu)$, the cubical complex on the grid, the complement of the grid, and the cubical complex on the complement.}
	\label{fig_gridcompex}
\end{figure}
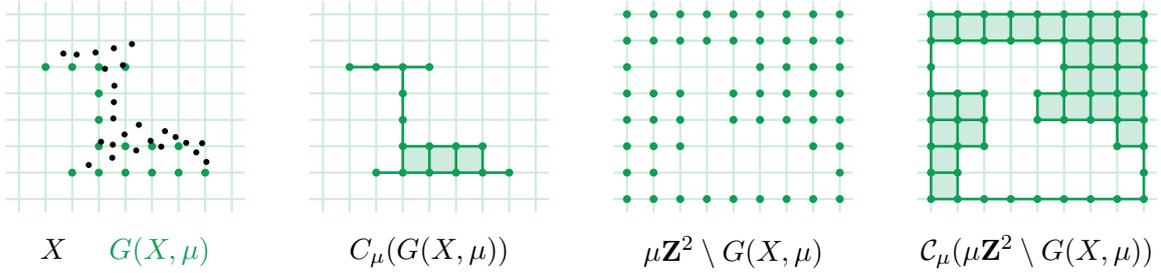

Given a grid $G(X,\mu)\subseteq \mu \Z^N$, the \emphasize{subdivision} $G_s(X,\mu)$ is the subset of $\frac\mu2\Z^N$ in which each element of $G(X,\mu)$ has been replaced by $2^N$ elements, by subdividing the $N$-cube of side length $\mu$ into $2^N$ $N$-cubes of side length $\frac\mu2$.
We also define the \emphasize{thickening} $G_t(X,\mu)$ as all elements in $\frac\mu2\Z^N$ at most a $d_\infty$-distance (in $\R^N$) of $\frac\mu2$ away from $G(X,\mu)$
As sets,
\begin{align}
    \label{eqn:gridsubdivision}
    G_s(X,\mu) & \colonequals \textstyle \bigcup_{\genfrac{}{}{0pt}{3}{x\in G(X,\mu)}{y\in \{0,1\}^N}} \left\{x+ \frac{\mu y_1}{2}e_1 + \cdots + \frac{\mu y_N}{2}e_N \right\}, \\
    \label{eqn:gridthickening}
    G_t(X,\mu) & \colonequals \left\{ x\in {\textstyle\frac\mu2}\Z^N\ :\ \text{there exists\ } x'\in G(X,\mu) \text{\ with\ } d_\infty(x,x') \leqslant {\textstyle\frac\mu2}\right\},
\end{align}
where $e_i$ is the $i$th standard basis coordinate vector (all zeros except for 1 in the $i$th coordinate).
A visual example of these constructions given in \cref{fig_gridsubdivisions}.

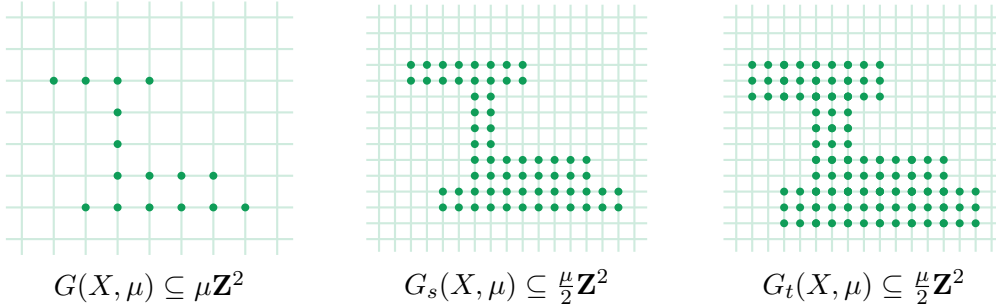
\begin{figure}[htbp]
    \centering
	\newcommand\pdist{.04} 
\newcommand\gdist{1.5pt} 
\newcommand\framestep{4.7} 
\newcommand\unit{.42} 
\begin{tikzpicture}[
  gridline/.style={bg,line width=.8pt},
  shapeline/.style={fg,line width=1pt}
]
\foreach \x in {-1,...,7}{\draw[gridline] (\x*\unit,-.5*\unit)--(\x*\unit,7.5*\unit);}
\foreach \y in {0,...,7}{\draw[gridline] (-1.5*\unit,\y*\unit)--(7.5*\unit,\y*\unit);}
\foreach \x\y in {0/5, 1/5, 2/5, 3/5, 2/4, 2/3, 2/2, 3/2, 4/2, 5/2, 1/1, 2/1, 3/1, 4/1, 5/1, 6/1}{
  \fill[fg] (\x*\unit,\y*\unit) circle[radius=\gdist];
}
\node at (3*\unit,-1.5*\unit) {$G(X,\mu)\subseteq \mu\Z^2$};
\begin{scope}[shift={(\framestep,0)}]
\foreach \x in {-1,...,7}{\draw[gridline] (\x*\unit,-.4*\unit)--(\x*\unit,7.4*\unit);}
\foreach \x in {-1,...,6}{\draw[gridline] (\x*\unit+.5*\unit,-.4*\unit)--(\x*\unit+.5*\unit,7.4*\unit);}
\foreach \y in {0,...,7}{\draw[gridline] (-1.4*\unit,\y*\unit)--(7.4*\unit,\y*\unit);}
\foreach \y in {0,...,6}{\draw[gridline] (-1.4*\unit,\y*\unit+.5*\unit)--(7.4*\unit,\y*\unit+.5*\unit);}
\foreach \x\y in {0/5, 1/5, 2/5, 3/5, 2/4, 2/3, 2/2, 3/2, 4/2, 5/2, 1/1, 2/1, 3/1, 4/1, 5/1, 6/1}{
  \fill[fg] (\x*\unit,\y*\unit) circle[radius=\gdist];
  \fill[fg] (\x*\unit,\y*\unit+\unit/2) circle[radius=\gdist];
  \fill[fg] (\x*\unit+\unit/2,\y*\unit) circle[radius=\gdist];
  \fill[fg] (\x*\unit+\unit/2,\y*\unit+\unit/2) circle[radius=\gdist];
}
\node at (3*\unit,-1.5*\unit) {$G_s(X,\mu)\subseteq \frac\mu2\Z^2$};
\end{scope}
\begin{scope}[shift={(2*\framestep,0)}]
\foreach \x in {-1,...,7}{\draw[gridline] (\x*\unit,-.4*\unit)--(\x*\unit,7.4*\unit);}
\foreach \x in {-1,...,6}{\draw[gridline] (\x*\unit+.5*\unit,-.4*\unit)--(\x*\unit+.5*\unit,7.4*\unit);}
\foreach \y in {0,...,7}{\draw[gridline] (-1.4*\unit,\y*\unit)--(7.4*\unit,\y*\unit);}
\foreach \y in {0,...,6}{\draw[gridline] (-1.4*\unit,\y*\unit+.5*\unit)--(7.4*\unit,\y*\unit+.5*\unit);}
\foreach \x\y in {0/5, 1/5, 2/5, 3/5, 2/4, 2/3, 2/2, 3/2, 4/2, 5/2, 1/1, 2/1, 3/1, 4/1, 5/1, 6/1}{
  \foreach \xshift in {-1,0,1}{
    \foreach \yshift in {-1,0,1}{
      \fill[fg] (\x*\unit+\xshift*\unit*.5,\y*\unit+\yshift*\unit*.5) circle[radius=\gdist];
    }
  }
}
\node at (3*\unit,-1.5*\unit) {$G_t(X,\mu)\subseteq \frac\mu2\Z^2$};
\end{scope}
\end{tikzpicture}
	\caption{From left to right: A subset $G(X,\mu)$ of a $\mu$-grid (taken from \cref{fig_gridcompex}, left), the subdivision of the subset, and the thickening of the subset.}
	\label{fig_gridsubdivisions}
\end{figure}


We may constructively define the thickening by taking the $2^N$ copies of the subdivision $G_s(X,\mu)$, each of which has been shifted by a linear combination of the standard basis vectors $e_i$, where the shift is determined by an element of $\{0,-\frac12\}^N$.


\begin{remark}\label{rem_gridmap}
Given a subset $X\subseteq \R^N$, another way to consider the $(\mu,z)$-grid $G(X,\mu)$ is as the subset of $\Z^N$ scaled by $\mu$, for which $x'\in G(X,\mu)$ whenever there is some $x\in X$ such that $x-x'\in [0,\mu)^N$.
This defines a unique map $X\to \mu\Z^N$, whose image is precisely $G(X,\mu)$.
This map will be instrumental in proving \cref{thm_complement}.
\end{remark}

\begin{remark}\label{rem_gridconstructions}
Before presenting our main results, we make some final remarks about the construction choices made.
\begin{itemize}
    \item The presented construction of a grid $G(X,\mu)$ from a set $X\subseteq \R^N$ follows one of many constructions, and is natural perhaps only by computational considerations, as computing this grid is done by division without a remainder. 
    Another common construction would be to find the nearest element in $\mu\Z^N$ using the Euclidean metric, and we note that this is not equivalent to our construction.
    \item Given a finite metric space $(X,d)$ and $X'\subseteq X$, the Vietoris--Rips filtration on $X$ is always a refinement (maintains the order of spaces) of the Vietoris--Rips filtration on $X'$.
    This is also true for cubical complexes when $G,G'\subseteq \mu\Z^2$, but it is not true for the alpha complex.
\end{itemize}
\end{remark}

\section{Stability}
\label{sec_mainresults}

Now we describe the effect on 0-dimensional persistent homology of adding new elements and removing existing elements from a finite set in Euclidean space.

\subsection{Bounding the bottleneck distance}
Our first three results concern the bottleneck distance between the 0-dimensional persistence diagram of the Vietoris--Rips filtration on two related point clouds.
The first result is about the stability of the barycentric subdivision method $B(X,\delta)$.

\begin{theorem}\label{thm_mainbary}
    Let $X\subseteq \R^N$ be a finite set, and $B(X,\delta) \supseteq X$ the vertices of the barycentric subdivision of 1-simplices and 2-simplices of $VR_\delta(X)$. Let $D_X,D_{B(X,\delta)}$ be persistence diagrams of the Vietoris--Rips filtrations in degree 0 on $X$ and $B(X,\delta)$, respectively. Then
    \begin{equation}
    \distbottleneck\left( D_X,D_{B(X,\delta)} \right) \leqslant \frac{\delta}{4}.
    \end{equation}
\end{theorem}

\begin{proof}
There is a natural inclusion $\varphi\colon D_X \hookrightarrow D_{B(X,\delta)}$, given by (the inverse of) $\erule_X$ and $\erule_{B(X,\delta)}$ from \cref{def_points2bars}, and by the inclusion $\iota$ of $X$ into $B(X,\delta)$.
This map is represented as the vertical dashed line in the commutative diagram
\begin{equation}
    \label{diag_barydiag}
    \newcommand\hstretch{4}
\newcommand\vstretch{1.5}
\begin{tikzpicture}[baseline=-2pt]
\node (x) at (0,\vstretch) {$X$};
\node (bx) at (0,0) {$B(X,\delta)$};
\node (dx) at (\hstretch,\vstretch) {$D_X$};
\node (dbx) at (\hstretch,0) {$D_{B(X,\delta)}$};
\draw[->] (x) to node[above] {$\erule_X$} node[below] {$\simeq$} (dx);
\draw[->] (bx) to node[above] {$\erule_{B(X,\delta)}$} node[below] {$\simeq$} (dbx);
\draw[{Hooks[right]}->] (x) to node[left] {$\iota$} (bx);
\draw[{Hooks[right]}->,dashed] (dx) to node[right] {$\varphi$} (dbx);
\end{tikzpicture}.
\end{equation}
Consider an interval in $D_{B(X,\delta)}\setminus \varphi(D_X)$.
If this interval is associated with the midpoint of a 1-simplex, then it must be a subinterval of $[0,\delta/4)$.
This follows from the definition of the barycentric subdivision at a distance $\delta$ and as the points in $X$ are indexed lower than the points in $B(X,\delta)\setminus X$.
If this interval is associated with the centroid of a 2-simplex, then it must be a subinterval of $[0,\delta/4\sqrt{3})$, as the centroid of a triangle with side lengths no longer than $\delta$ is at most $\delta/2\sqrt{3}$ away from any one of the midpoints of the triangle's edges.
These configurations are visually presented in \cref{fig_baryproof}.

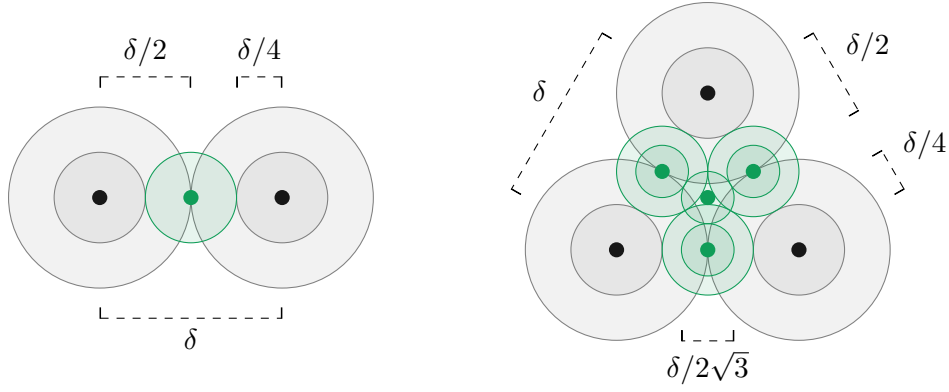
\begin{figure}[htbp]
    \centering
	\newcommand\hstep{1.2} 
\newcommand\rdist{.8} 
\newcommand\labdist{1.6} 
\begin{tikzpicture}[
  reg/.style={draw=black!50},
  regf/.style={fill=gray,fill opacity=.1},
  greg/.style={draw=fg},
  gregf/.style={fill=fg,fill opacity=.1},
  brr/.style={dashed,{Rectangle[right,length=.4pt,width=8pt]}-{Rectangle[left,length=.4pt,width=8pt]}},
  brl/.style={dashed,{Rectangle[left,length=.4pt,width=8pt]}-{Rectangle[right,length=.4pt,width=8pt]}}
]
\foreach \x\n\c in {0/a/black, 2*\hstep/b/black}{
  \coordinate (\n) at (\x,0);
  \fill[\c] (\n) circle (.1);
  \draw[reg,regf] (\n) circle (\hstep/2);
}
\draw[reg,regf] (a) circle (\hstep);
\draw[reg,regf] (b) circle (\hstep);
\foreach \x\n\c in {\hstep/c/fg}{
  \coordinate (\n) at (\x,0);
  \fill[\c] (\n) circle (.1);
  \draw[greg,gregf] (\n) circle (\hstep/2);
}
\draw[brr]  ($(a)+(270:\labdist)$) to node[below] {$\delta$} ($(b)+(270:\labdist)$);
\draw[brl]  ($(a)+(90:\labdist)$) to node[above] {$\delta/2$} ($(c)+(90:\labdist)$);
\draw[brr]  ($(b)+(90:\labdist)$) to node[above] {$\delta/4$} ++(180:\hstep/2);
\begin{scope}[shift={(8,0)}]
\foreach \r\n in {90/a, 210/b, 330/c}{
  \coordinate (\n) at ($(0,0)+(\r:2*.5774*\hstep)$);
  \fill[black] (\n) circle (.1);
  \draw[reg,regf] (\n) circle (\hstep);
}
\foreach \r\n in {150/d, 270/e, 30/f}{
  \coordinate (\n) at ($(0,0)+(\r:.5774*\hstep)$);
}
\coordinate (g) at (0,0);
\foreach \n in {a,b,c}{
  \draw[reg,regf] (\n) circle (\hstep/2);
}
\foreach \n in {d,e,f}{
  \draw[greg,gregf] (\n) circle (\hstep/2);
}
\foreach \n in {d,e,f,g}{
  \fill[fg] (\n) circle (.1);
  \draw[greg,gregf] (\n) circle (.2887*\hstep);
}
\draw[brr]  ($(a)+(150:\labdist)$) to node[auto,swap] {$\delta$} ($(b)+(150:\labdist)$);
\draw[brl]  ($(a)+(30:\labdist)$) to node[auto] {$\delta/2$} ($(f)+(30:\labdist)$);
\draw[brr]  ($(c)+(30:\labdist)$) to node[auto,swap] {$\delta/4$} ++(120:\hstep/2);
\draw[brl]  (.2887*\hstep,-1.9) to node[below] {$\delta/2\sqrt{3}$} (-.2887*\hstep,-1.9);
\end{scope}
\end{tikzpicture}
	\caption{Configurations for the barycentric subdivisions of 1-simplices (midpoints, left) and 2-simplices (centroids, right). Elements of $X$ are drawn black, and elements of $B(X,\delta)\setminus X$ are drawn green.}
	\label{fig_baryproof}
\end{figure} 

Next, consider an interval in $D_X$ and its image in $\varphi(D_X)$. The interval in the image cannot be longer than it was before, as the set $B(X,\delta)$ fully contains $X$. If the interval is shorter, it can be shorter by at most $\delta/4$, due to a new point at the midpoint between it and another point. This follows by observing that for $x_a,x_b\in X$ with $d(x_a,x_b)<\delta$ and $x_c = (x_a+x_b)/2$, the ball centered at $x_c$ with radius $r$ is completely covered by the balls centered at $x_a$ and $x_b$ of radius $r+\delta/2$, as described in \cref{fig_baryproof}.

The map $\varphi$ may be considered as a bijection, after enriching both the source and the target with countably many points $(r,r)$ for every $r\in\R_{\geqslant 0}$ (a standard method in PH), and sending all elements not in the image of $\varphi$ to $(0,0)$. Hence for the map $\varphi\colon D_X\to D_{B(X,\delta)}$ the supremum over all the differences between an interval in $D_X$ and its image under $\varphi$ is at most $\delta/4$. It follows that $\distbottleneck(D_X,D_{B(X,\delta)}) \leqslant \delta/4$.
\end{proof}

The second result is about the stability of the sparsification method $S(X,\varepsilon)$.

\begin{theorem}\label{thm_mainsparse}
    Let $X\subseteq \R^N$ be a finite set, and $S(X,\epsilon) \subseteq X$ the sparsification of $X$ at a minimum distance $\varepsilon>0$.
    Let $D_X,D_{S(X,\varepsilon)}$ be the persistence diagrams of the Vietoris--Rips filtrations in degree 0 on $X$ and $S(X,\epsilon)$, respectively.
    Then
    \begin{equation}
    \distbottleneck(D_X,D_{S(X,\varepsilon)}) \leqslant \frac\varepsilon2.
    \end{equation}
\end{theorem}

\begin{proof}
Analogously to the proof of \cref{thm_mainbary}, there is a natural inclusion $\varphi\colon D_{S(X,\varepsilon)} \hookrightarrow D_X$, induced by the inclusion $\iota: S(X,\varepsilon) \hookrightarrow X$. More precisely, it may be constructed as the composition $\varphi = \erule_X \circ \iota \circ \erule_{S(X,\varepsilon)}^{-1}$, presented visually as the vertical dashed line in the commutative diagram
\begin{equation}
    \label{diag_sparsediag}
    \newcommand\hstretch{4}
\newcommand\vstretch{1.5}
\begin{tikzpicture}[baseline=-2pt]
\node (x) at (0,\vstretch) {$S(X,\varepsilon)$};
\node (bx) at (0,0) {$X$};
\node (dx) at (\hstretch,\vstretch) {$D_{S(X,\varepsilon)}$};
\node (dbx) at (\hstretch,0) {$D_X$};
\draw[->] (x) to node[above] {$\erule_{S(X,\varepsilon)}$} node[below] {$\simeq$} (dx);
\draw[->] (bx) to node[above] {$\erule_X$} node[below] {$\simeq$} (dbx);
\draw[{Hooks[right]}->] (x) to node[left] {$\iota$} (bx);
\draw[{Hooks[right]}->,dashed] (dx) to node[right] {$\varphi$} (dbx);
\end{tikzpicture}.
\end{equation}
Consider an interval in $D_X \setminus D_{S(X,\varepsilon)}$.
The interval must be a subinterval of $[0,\varepsilon/2)$, as $x_i\in X \setminus S(X,\varepsilon)$ implies that there is $x_j\in S(X,\varepsilon)$ with $i>j$ and $d(x_i,x_j)< \varepsilon$.
That is, by the method of sparsification used (\cref{def_sparsification}), the 0-class born at $x_i$ must die no later than $\varepsilon/2$.

\begin{figure}[htbp]
    \centering
	\newcommand\bigsp{7} 
\newcommand\minsp{0} 
\begin{tikzpicture}[
  reg/.style={draw=black!50},
  regf/.style={fill=gray,fill opacity=.1},
  greg/.style={draw=fg},
  gregf/.style={fill=fg,fill opacity=.1},
  brr/.style={dashed,{Rectangle[right,length=.4pt,width=8pt]}-{Rectangle[left,length=.4pt,width=8pt]}},
  brl/.style={dashed,{Rectangle[left,length=.4pt,width=8pt]}-{Rectangle[right,length=.4pt,width=8pt]}}
]
\path[clip] (-1.5,-1.5) rectangle (\bigsp+1.5,1.5);
\foreach \x\n\l\anch in {0/i/i/90, 1/k/{k>i}/135, \bigsp/j/{j>i}/90}{
  \coordinate (\n) at (\x,0);
  \node[anchor=\anch,inner sep=8pt] at (\n) {$x_{\l}$};
}
\draw[greg,gregf] (k) circle (\bigsp/2-.5);
\foreach \n\r in {i/1, i/{\bigsp/2}, j/{\bigsp/2}, j/{\bigsp/2-.5}}{
  \draw[reg,regf] (\n) circle (\r);
}
\foreach \x\n\l\anch in {0/i/i/90, \bigsp/j/{j>i}/90}{
  \fill[black] (\x,0) circle (.1);
}
\fill[fg] (k) circle (.1);
\draw[brr] ($(i)+(270:.9)$) to node[below] {$r$} ++ (0:\bigsp/2-\minsp);
\draw[brl] ($(j)+(270:.9)$) to node[below] {$r$} ++ (180:\bigsp/2-\minsp);
\draw[brl] ($(i)+(90:.7)$) to node[above] {$\varepsilon$} ++ (0:1-\minsp);
\draw[brl] ($(k)+(90:.7)$) to node[above] {$r-\varepsilon/2$} ++ (0:\bigsp/2-.5-\minsp);
\draw[brr] ($(j)+(90:.7)$) to node[above] {$r-\varepsilon/2$} ++ (180:\bigsp/2-.5-\minsp);
\end{tikzpicture}
	\caption{Configuration for the effects of sparsification on $X$. Elements of $S(X,\epsilon)$ are drawn black, and elements of $X\setminus S(X,\epsilon)$ are drawn green.}
	\label{fig_sparseproof}
\end{figure}
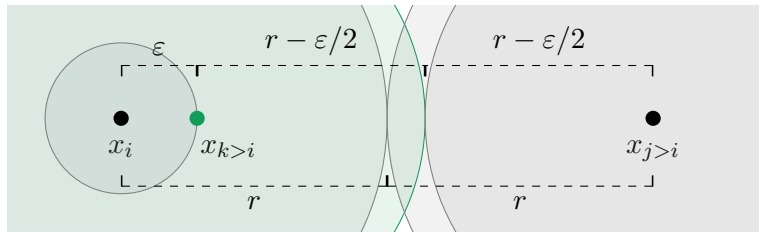 

Next, consider an interval in $D_{S(X,\varepsilon)}$ and its image in $\varphi(D_{S(X,\varepsilon)})$. 
The interval in the image cannot be longer than it was before, as the set $X$ fully contains $S(X,\varepsilon)$.
If the interval is shorter, it may be shorter by at most $\varepsilon/2$, as the class that kills it in $D_{S(X,\varepsilon)}$ may reach the interval at most $\varepsilon/2$ earlier, from a point within $\varepsilon$ of that class.
This is demonstrated in \cref{fig_sparseproof}.
As in the proof of \cref{thm_mainbary}, the map $\varphi$ may be considered a bijection, and so $\distbottleneck(D_X,D_{S(X,\varepsilon)}) \leqslant \varepsilon/2$.
\end{proof}

The third result is about the stability of the grid method $G(X,\mu)$.

\begin{theorem}\label{thm_maingrid}
    Let $X\subseteq \R^N$ be a finite set and $G(X,\mu)$ be a $\mu$-grid of $X$.
    Let $D_X$ be the persistence diagrams of the Vietoris--Rips filtration in degree 0, and $D_{G(X,\mu)}$ the persistence diagram of the cubical filtration in degree 0.
    Then
    \begin{equation}
        \distbottleneck\left( D_X,D_{G(X,\mu)} \right) \leqslant \frac{\sqrt N\mu}{2}.
    \end{equation}
\end{theorem}

\begin{proof}
Note that there is no inclusion from $X$ into $G(X,\mu)$, as the map $g\colon X\to G(X,\mu)$ described in \cref{rem_gridmap} is not necessarily injective.
However, we can still construct a map $\varphi\colon D_X\to D_{G(X,\mu)}$ using $g$, by defining
\begin{equation}
\varphi(c) = \begin{cases}
\tilde{\varphi}(c) & \text{\ if\ $c$ is the longest interval in $\tilde{\varphi}^{-1}(c)$, } \\
0 & \text{\ else,}
\end{cases}
\end{equation}
where $\tilde{\varphi} = \erule_{G(X,\mu)}  \circ g \circ \erule_X^{-1}$ is an analogous map to the induced $\varphi$ in \cref{thm_mainbary,thm_mainsparse}.
With this, we will compute the distance $\distbottleneck(D_X,D_{G(X,\mu)})$ to the claimed bound. 
Note that $\varphi$ is surjective (when restricted to the support in the source and target), so it will suffice to consider the image of every element in $D_X$ under $\varphi$ to compute this distance.

First consider an interval $c\in D_X$ for which $\varphi(c) = 0$.
For $x=\erule_X^{-1}(c)$, this means that there is some $y\in X$ with $g(x)=g(y)$, and $y$ ordered before $x$.
As $d(x,y)<\sqrt N\mu$, since $\sqrt N$ is the longest distance between any two points in the $N$-cube $[0,1]^N$, it follows that $c\subseteq [0,\sqrt N\mu/2)$.
Hence in this case, it means the difference between $c$ and $\varphi(c)$ is bounded above by $\sqrt N\mu/2$.

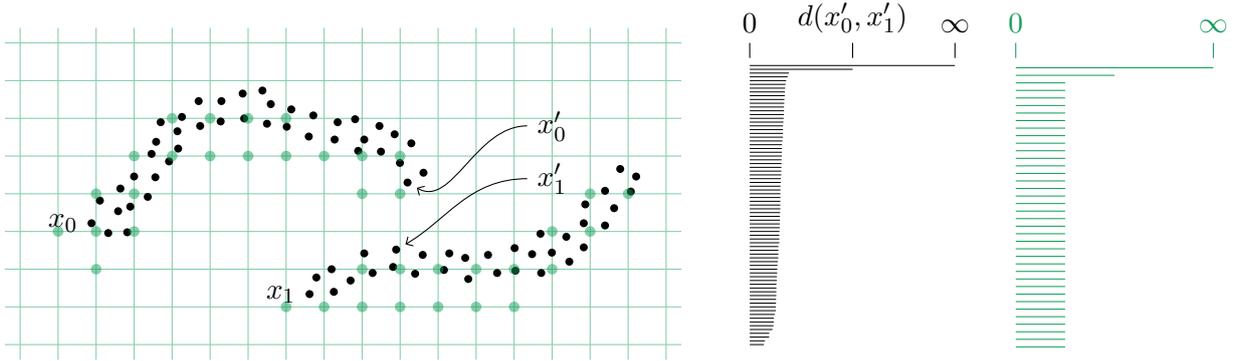
\begin{figure}[htbp]
    \centering
	\newcommand\unit{.5} 
\newcommand\overlap{.25} 
\begin{tikzpicture}[
  reg/.style={line width=1pt,draw=black!50},
  lab/.style={inner sep=7pt},
  alab/.style={fill=white,fill opacity=.8,text opacity=1},
  brr/.style={line width=1pt,dashed,{Rectangle[right,length=1pt,width=10pt]}-{Rectangle[left,length=1pt,width=10pt]}},
  brl/.style={line width=1pt,dashed,{Rectangle[left,length=1pt,width=10pt]}-{Rectangle[right,length=1pt,width=10pt]}}
]
\foreach \x in {0,...,17}{\draw[ggreen!50] (\x*\unit-.1,-.3)--(\x*\unit-.1,4.1);}
\foreach \y in {0,...,8}{\draw[ggreen!50] (-.3,\y*\unit-.1)--(8.6,\y*\unit-.1);}
\foreach \x\y\n in {
106/138/6, 119/167/7, 84/151/8, 131/139/9, 135/173/10, 158/186/11, 186/233/12, 95/181/15, 140/213/16, 169/259/17, 203/292/18, 227/280/19, 772/171/20, 197/273/21, 255/313/24, 285/290/25, 254/286/26, 225/313/27, 175/285/28, 283/323/29, 760/148/30, 408/285/32, 404/262/33, 341/279/35, 376/294/36, 347/302/37, 320/309/38, 465/246/40, 431/289/42, 435/247/43, 370/266/45, 443/262/46, 483/269/47, 463/280/49, 500/205/50, 490/231/51, 521/218/52, 505/257/54,  380/79/56, 400/90/57, 371/57/58, 425/75/61, 443/111/62, 485/116/63, 520/109/64, 454/85/65, 548/89/67, 576/105/68, 555/111/69, 641/118/72, 642/88/73, 676/85/74, 713/100/75, 690/112/77, 664/110/78, 732/119/79, 709/133/83, 675/137/84, 732/151/85, 309/327/86, 122/197/87, 760/194/89, 801/213/90, 794/193/95, 734/176/99, 780/223/100,  315/283/109, 168/212/110, 198/250/111, 163/243/112, 403/60/115, 481/93/117, 580/77/118, 618/85/119, 510/84/121, 606/109/122}{
  \coordinate (\n) at (.01*\x,.01*\y);
  \fill[black] (\n) circle[radius=1.5pt];
}
\node[anchor=east,inner sep=5pt] at (8) {$x_0$};
\node[anchor=east,inner sep=5pt] at (58) {$x_1$};
\node (x0p) at (6.9,2.8) {$x_0'$};
\node (x1p) at (6.9,2.1) {$x_1'$};
\draw[-{Straight Barb[width=4pt,length=2pt]},shorten >=4pt] (x0p.west) to [out=180,in=330] (50);
\draw[-{Straight Barb[width=4pt,length=2pt]},shorten >=4pt] (x1p.west) to [out=180,in=30] (63);
\foreach \x\y in {1/3, 2/2, 2/3, 2/4, 3/3, 3/4, 3/5, 4/5, 4/6, 5/5, 5/6, 6/5, 6/6, 7/5, 7/6, 8/5, 9/4, 9/5, 10/4, 10/5, 7/1, 8/1, 9/1, 9/2, 10/1, 10/2, 11/1, 11/2, 12/1, 12/2, 13/1, 13/2, 14/2, 14/3, 15/3, 15/4, 16/4}{
  \fill[ggreen,opacity=.5] (\x*\unit-.1,\y*\unit-.1) circle[radius=2pt];
}
\begin{scope}[shift={(9.5,-.1)}]
\setcounter{mycounter}{0}
\foreach \x in {12.3693, 13.3417, 17, 17.088, 19.9249, 20.8087, 21.095, 21.1896, 21.8403, 22.8254, 22.8254, 23.0217, 23.0868, 23.2594, 23.3452, 23.3452, 23.3452, 23.7697, 23.7697, 24.0832, 24.1868, 24.3516, 24.6982, 25.0599, 25.0599, 25.0799, 25.0799, 25.9422, 25.9422, 25.9422, 26.0768, 26.3059, 26.3059, 26.4197, 26.4764, 26.6271, 26.8328, 26.8328, 26.8701, 26.9072, 26.9258, 26.9258, 27.0185, 27.2029, 27.2947, 27.6586, 27.6586, 27.7308, 27.8568, 27.8568, 28.1603, 28.1603, 28.2312, 28.2843, 28.6356, 29.1548, 29.1548, 29.1548, 29.2062, 29.5466, 29.7321, 30, 30.0167, 30.0832, 30.2655, 30.3645, 30.6757, 30.8058, 31.1127, 31.3847, 31.9531, 33.2415, 34.4093, 90.2552, 180}{
  \draw[line width=.2pt] (0,.05*\themycounter)--(.015*\x,.05*\themycounter);
  \stepcounter{mycounter}
}
\foreach \x\l in {0/0, {90.2552*.015}/{d(x_0',x_1')}, {180*.015}/{\infty}}{
  \draw (\x,76*.05) --++ (90:.2) node[above] {$\l$};
}
\end{scope}
\begin{scope}[shift={(13,-.1)}]
\setcounter{mycounter}{0}
\foreach \x in {1, 1, 1, 1, 1, 1, 1, 1, 1, 1, 1, 1, 1, 1, 1, 1, 1, 1, 1, 1, 1, 1, 1, 1, 1, 1, 1, 1, 1, 1, 1, 1, 1, 1, 1, 1, 2, 4}{
  \draw[ggreen,line width=.2pt] (0,.1*\themycounter-.03)--(.65*\x,.1*\themycounter-.03);
  \stepcounter{mycounter}
}
\foreach \x\l in {0/0, {4*.65}/{\infty}}{
  \draw[ggreen] (\x,76*.05) --++ (90:.2) node[ggreen,above] {$\l$};
}
\end{scope}
\end{tikzpicture}
	\caption{An example of a finite set $X\subseteq \R^2$ (left, black) and its image $G(X,\mu)$ in a grid (left, green).
    Their persistence diagrams in dimension 0 are also compared (right).}
	\label{fig_gridproof}
\end{figure} 

Next consider an interval $c\in D_X$ for which $\varphi(c)=\tilde\varphi(c)$.
If $c= [0,\infty)$, it must also be that $\varphi(c)=[0,\infty)$, in which case there is nothing else to prove, so assume that $c\neq [0,\infty)$.
For $x=\erule_X^{-1}(c)$, consider $y\in X$ such that $c$ is killed by $\erule_X(y)$, and killed by $(y',x')\in X^2$.
Similarly, suppose that $\varphi(c)$ is killed by $(y'',x'')\in G(X,\mu)^2$.
Let $\alpha,\beta\in [0,\mu)^N$ such that $x'=g(x')+\alpha$ and $y'=g(y')+\beta$, for which
\begin{align*}
d(g(x'),g(y')) & = \Arrowvert g(x') - g(y')\Arrowvert \\
& = \Arrowvert (x'-\alpha) - (y'-\beta)\Arrowvert \\
& = \Arrowvert (x'-y') + (\beta - \alpha) \Arrowvert \\
& \leqslant \Arrowvert x'-y' \Arrowvert  + \Arrowvert \beta - \alpha \Arrowvert \\
& = d(x',y') + d(\alpha,\beta) \\
& \leqslant d(x',y') + \sqrt N\mu.
\end{align*}
With this, and knowing $d(x'',y'') \leqslant d(g(x'),g(y'))$ by definition, we get 
\begin{equation}
d(x'',y'') \leqslant d(x',y')+\sqrt N \mu
\ \ \implies\ \ 
|d(x'',y'') - d(x',y')| \leqslant \sqrt N\mu.
\end{equation}
Hence also in this case, the difference between $c$ and $\varphi(c)$ is bounded above by $\sqrt N\mu/2$.  
\end{proof}

\begin{remark}
The bounds described in \cref{thm_mainbary,thm_mainsparse} are strict, and may be achieved by the configurations of small point clouds presented in \cref{fig_baryproof,fig_sparseproof}, respectively.
The bound in \cref{thm_maingrid} is also strict, but as $\sqrt N$ is the diagonal distance in $[0,1]^N$, it is not achieved by any configuration.
Nonetheless, for any $\nu>0$, a bottleneck distance of $\frac{\sqrt N\mu - \nu}{2}$  is achieved by $X=\{\mathbf{0},(1-\frac{\nu}{\sqrt{N}})\mathbf{1}\}$ and $\mu=1$, where $\mathbf{0} \colonequals (0,\dots,0),\mathbf{1} \colonequals (1,\dots,1)\in \R^N$. 
\end{remark}

\subsection{Computing higher dimensions with duality}
\label{sec_cubicalduality}
Here we consider the relationship between the homology groups of a cubical complex on a grid, and a cubical complex on the complement of the grid.
The underlying point cloud is still finite, and homology is computed with coefficients in a field, so we may identify homology and cohomology groups.

\begin{theorem}
    \label{thm_complement}
    Let $X\subseteq \R^N$ be a finite set and $G(X,\mu)$ be a $\mu$-grid of $X$. Then
    \begin{equation}\label{eqn_complement}
    H_{N-1}(\mathcal C_{\mu/2}(G_t(X,\mu))) \oplus \Z_2 \simeq H_0(\mathcal C_\mu(\mu\Z^N \setminus G(X,\mu))).
    \end{equation}
\end{theorem} 

\begin{proof}
For ease of notation, let $\mathcal C' \colonequals \mathcal C_{\mu/2}(G_t(X,\mu))$ and $\mathcal C'' \colonequals \mathcal C_\mu(\mu\Z^N \setminus G(X,\mu))$, or, where appropriate, their geometric realizations embedded in $\R^N$.
Since $X$ is finite, there is an appropriately large open $N$-disk $D^N\subseteq \R^N$ with $\mathcal C' \subseteq D^N$.
Identifying the boundary of the closure of $D^N$ to a single point allows us to consider $\mathcal C' $ as a subset of $S^N$.
By Alexander duality, we have $\widetilde{H}_0(\R^N\setminus \mathcal C') \simeq \widetilde{H}_0(S^N\setminus \mathcal C') \simeq \widetilde{H}^{N-1}(\mathcal C')$, or equivalently $H^{N-1}(\mathcal C')$. As $\mathcal C'$ is a cubical complex on a finite set, its homology groups are finitely generated, so $H^{N-1}(\mathcal C')\simeq H_{N-1}(\mathcal C')$.

To complete the proof, we must show that $H_0(\R^N\setminus \mathcal C') \simeq H_0(\mathcal C'')$. 
This follows from a decomposition of $\R^N$ into disjoint sets.
Let
\begin{align*}
    X' & \colonequals \left\{ x\in \R^N\ :\ \text{there exists\ } x'\in G(X,\mu) \text{\ with\ } d_\infty(x,x') \leqslant {\textstyle\frac\mu2}\right\}, \\
    X'' & \colonequals  \left\{ x\in \R^N\ :\ \text{there exists\ } x'\in \mathcal C_\mu(\mu\Z^N \setminus G(X,\mu)) \text{\ with\ } d_\infty(x,x') < {\textstyle\frac\mu2}\right\}.
\end{align*}
It is immediate that $X'\cup X'' = \R^N$, as every point is within $\frac\mu2$ in each coordinate of an element of a $\mu$-grid.
Similarly, we see that $X'\cap X'' = \emptyset$, which comes from taking a grid on $\R^N$ and splitting into two disjoint sets.
Any element in the complement of the open $\frac\mu2$-neighbourhood of the cubical complex on $X'$ is within a distance of $\frac\mu2$, in every coordinate, of an element in $X''$.
Comparing with \cref{eqn:gridthickening}, we see that $X'$ is (the geometric realization of) $\mathcal C'$.
Finally, $X''$ is simply the open $\mu/2$-neighbourhood of (the geometric realization of) $\mathcal C''$, and as $\mathcal C''$ consists of $N$-cubes of side length $\mu$, every element in $X''$ has a unique closest element in $\mathcal C''$.
Hence $\mathcal C''$ is a deformation retract of $X''$.
Putting it all together, in homological dimension 0 we have 
\[
H_0(\mathcal C'') \simeq H_0(X'') \simeq H_0(\R^N\setminus X') \simeq H_0(\R^N\setminus \mathcal C'),
\]
as desired.
\end{proof}

\begin{figure}[htbp]
    \centering
    \newcommand\pdist{.04} 
\newcommand\gdist{1.5pt} 
\newcommand\framestep{3.2} 
\newcommand\unit{.3} 
\begin{tikzpicture}[
  baseline=2.6*\unit cm,
  gridline/.style={bg,line width=.8pt},
  shapeline/.style={fg,line width=1pt}
]
\foreach \z in {1,...,5}{
  \draw[gridline] (\z*\unit,.5*\unit)--(\z*\unit,5.5*\unit);
  \draw[gridline] (.5*\unit,\z*\unit)--(5.5*\unit,\z*\unit);
}
\foreach \x\y in {2/3, 3/2, 4/3, 3/4}{
  \fill[fg] (\x*\unit,\y*\unit) circle[radius=\gdist];
}
\node at (3*\unit,-1*\unit) {$G(X,\mu)$};
\begin{scope}[shift={(\framestep,0)}]
\fill[white] (2.5*\unit,2.5*\unit) rectangle (3.5*\unit,3.5*\unit);
\foreach \z in {1,...,5}{
  \draw[gridline] (\z*\unit,.6*\unit)--(\z*\unit,5.4*\unit);
  \draw[gridline] (.6*\unit,\z*\unit)--(5.4*\unit,\z*\unit);}
\foreach \z in {1,...,4}{
  \draw[gridline] (\z*\unit+.5*\unit,.6*\unit)--(\z*\unit+.5*\unit,5.4*\unit);
  \draw[gridline] (.6*\unit,\z*\unit+.5*\unit)--(5.4*\unit,\z*\unit+.5*\unit);}
\foreach \x\y in {2/3, 3/2, 4/3, 3/4}{
  \foreach \xshift in {-1,0,1}{
    \foreach \yshift in {-1,0,1}{
      \fill[fg] (\x*\unit+\xshift*\unit*.5,\y*\unit+\yshift*\unit*.5) circle[radius=\gdist];
    }
  }
}
\node at (3*\unit,-1*\unit) {$G_t(X,\mu)$};
\end{scope}
\begin{scope}[shift={(2*\framestep,0)}]
\fill[bg] (2.5*\unit,1.5*\unit) --++ (0:\unit) --++ (90:\unit) --++ (0:\unit) --++ (90:\unit) --++ (180:\unit) --++ (90:\unit) --++ (180:\unit) --++ (270:\unit) --++ (180:\unit) --++ (270:\unit) --++ (0:\unit) -- cycle;
\fill[white] (2.5*\unit,2.5*\unit) rectangle (3.5*\unit,3.5*\unit);
\foreach \z in {1,...,5}{
  \draw[gridline] (\z*\unit,.6*\unit)--(\z*\unit,5.4*\unit);
  \draw[gridline] (.6*\unit,\z*\unit)--(5.4*\unit,\z*\unit);}
\foreach \z in {1,...,4}{
  \draw[gridline] (\z*\unit+.5*\unit,.6*\unit)--(\z*\unit+.5*\unit,5.4*\unit);
  \draw[gridline] (.6*\unit,\z*\unit+.5*\unit)--(5.4*\unit,\z*\unit+.5*\unit);}
\foreach \x\y in {2/3, 3/2, 4/3, 3/4}{
  \foreach \xshift in {-1,0,1}{
    \foreach \yshift in {-1,0,1}{
      \fill[fg] (\x*\unit+\xshift*\unit*.5,\y*\unit+\yshift*\unit*.5) circle[radius=\gdist];
    }
  }
}
\foreach \x\y\ang in {2.5/1.5/90, 3.5/1.5/90, 1.5/2.5/0, 1.5/3.5/0}{
  \draw[shapeline] (\x*\unit,\y*\unit) -- ++ (\ang:3*\unit);}
\foreach \x\y\ang in {3/1.5/90, 3/3.5/90, 1.5/2.5/90, 4.5/2.5/90, 2/2.5/90, 4/2.5/90, 1.5/3/0, 3.5/3/0, 2.5/1.5/0, 2.5/4.5/0, 2.5/2/0, 2.5/4/0}{
  \draw[shapeline] (\x*\unit,\y*\unit) -- ++ (\ang:\unit);}
\node at (3*\unit,-1*\unit) {$\mathcal C_{\mu/2}(G_t(X,\mu))$};
\end{scope}
\begin{scope}[shift={(3*\framestep,0)}]
\foreach \z in {1,...,5}{
  \draw[gridline] (\z*\unit,.4*\unit)--(\z*\unit,5.4*\unit);
  \draw[gridline] (.4*\unit,\z*\unit)--(5.4*\unit,\z*\unit);}
\foreach \x\y in {1/1, 1/2, 1/3, 1/4, 1/5, 2/1, 2/2, 2/4, 2/5, 3/1, 3/3, 3/5, 4/1, 4/2, 4/4, 4/5, 5/1, 5/2, 5/3, 5/4, 5/5}{
  \fill[fg] (\x*\unit,\y*\unit) circle[radius=\gdist];
}
\node at (3*\unit,-1*\unit) {$\mu\Z \setminus G(X,\mu)$};
\end{scope}
\begin{scope}[shift={(4*\framestep,0)}]
\foreach \x\y in {1/1, 4/1, 1/4, 4/4}{
  \fill[bg] (\x*\unit,\y*\unit) rectangle (\x*\unit+\unit,\y*\unit+\unit);}
\foreach \z in {1,...,5}{
  \draw[gridline] (\z*\unit,.4*\unit)--(\z*\unit,5.4*\unit);
  \draw[gridline] (.4*\unit,\z*\unit)--(5.4*\unit,\z*\unit);}
\foreach \x\y in {1/1, 1/2, 1/3, 1/4, 1/5, 2/1, 2/2, 2/4, 2/5, 3/1, 3/3, 3/5, 4/1, 4/2, 4/4, 4/5, 5/1, 5/2, 5/3, 5/4, 5/5}{
  \fill[fg] (\x*\unit,\y*\unit) circle[radius=\gdist];
}
\foreach \x\y\ang in {1/1/90, 5/1/90, 1/1/0, 1/5/0}{
  \draw[shapeline] (\x*\unit,\y*\unit) -- ++ (\ang:4*\unit);}
\foreach \x\y\ang in {2/1/90, 4/1/90, 2/4/90, 4/4/90, 1/2/0, 4/2/0, 1/4/0, 4/4/0}{
  \draw[shapeline] (\x*\unit,\y*\unit) -- ++ (\ang:\unit);}
\node at (3*\unit,-1*\unit) {$\mathcal C_\mu(\mu\Z \setminus G(X,\mu))$};
\end{scope}
\end{tikzpicture}
	\caption{A visual example of the spaces discussed in \cref{thm_complement}. Note that $G(X,\mu)$ and $\mu\Z\setminus G(X,\mu)$ are complementary as grids, while $\mathcal C_{\mu/2}(G_t(X,\mu))$ and $C_\mu(\mu\Z\setminus G(X,\mu))$ are complementary (up to homotopy) as spaces.}
	\label{fig_examples4}
\end{figure}
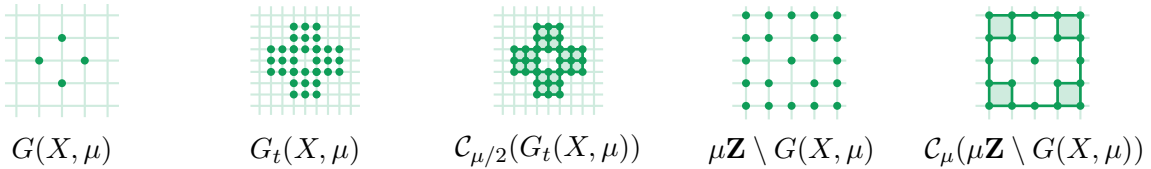

\section{Implementation and testing}
\label{sec_testing}

The data and code presented in this section is on a publicly available GitHub repository \cite{Lazovskis_TopoAware_2024}, which we call ``{\sc\nameshort}: \namelong".
The functions which we provide, in C++, Python, and R, are constructed following the constructions of \cref{sec_structures}, for which a point cloud is both an input and output.
Our contribution is to provide an easy-to-use collection of functions, for immediate use in computational pipelines or exploratory notebooks.
The main functions are presented in  \cref{table_functions}.

\begin{table}[hbtp]
    \centering
    \renewcommand\arraystretch{1.4}
    \begin{tabular}{l|l|l|l}
    function & input & output & guarantee \\\hline
    \texttt{barycentric\_subdivision} & $X\subseteq \R^N$, $\delta\in \R_{\geqslant 0}$ & $B(X,\delta) \subseteq \R^N$ & \cref{thm_mainbary} \\
    \texttt{sparsification} & $X\subseteq \R^N$, $\epsilon\in\R_{\geqslant 0}$ & $S(X,\epsilon) \subseteq \R^N$ & \cref{thm_mainsparse} \\
    \texttt{gridification} & $X\subseteq \R^N$, $\mu\in \R_{>0}$, $z\in \R^N$ & $G(X,\mu)\subseteq \mu\Z^N$ & \cref{thm_maingrid} \\
    \texttt{complement} &  $G\subseteq \mu\Z^N$ & $\mu\Z^N \setminus G\ \subseteq \mu\Z^N$ & \cref{thm_complement} \\
    \texttt{thickening} &  $G\subseteq \mu\Z^N$ & $G_t \subseteq \frac\mu2\Z^N$ & \cref{thm_complement}
    \end{tabular}
	\caption{Selected functions provided in the implementations.}
	\label{table_functions}
\end{table}



\subsection{Testing on synthetic data}
\label{sec_syntheticdata}

To demonstrate the effect of our methods and stability guarantees, we sample data from a known 2-dimensional shape in $\R^2$ (see \cref{fig_syntheticdata}, left).
Knowing the homology of the underlying shape from which the dataset is sampled allows us to better interpret the computed results.
The sample size is 1489, with sample density varying across the shape, and noise added to each sampled value.

\begin{figure}[htbp]
    \centering
	\includegraphics[width=\textwidth]{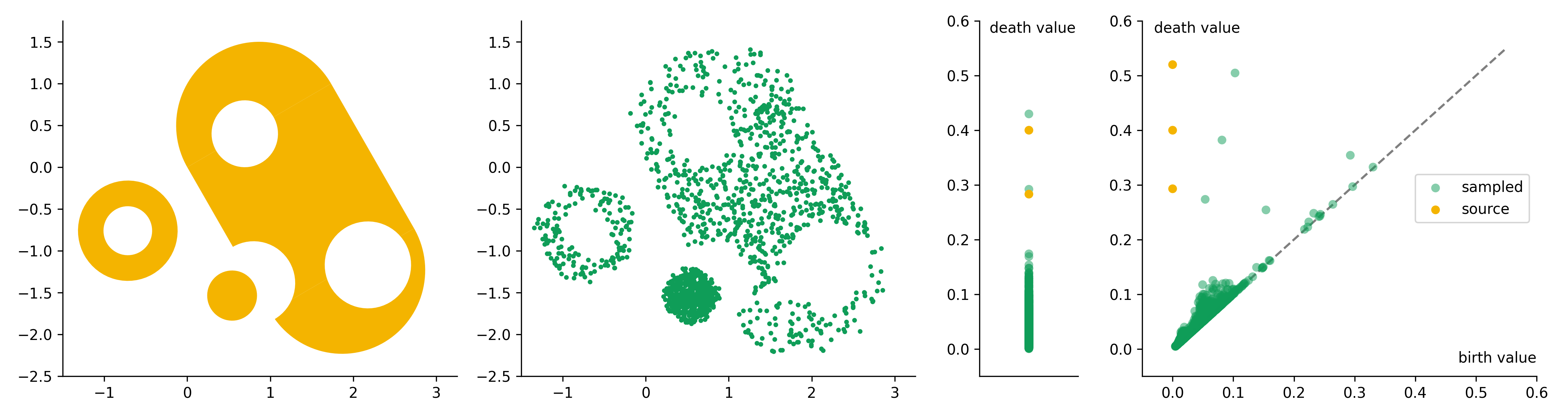}
	\caption{Source shape (left) and a sample of varying density of approximately 1500 points (center left). The associated finite classes in $D_0$ (center right) and in $D_1$ (right).}
	\label{fig_syntheticdata}
\end{figure} 

Given the sampled data, we compute the Vietoris--Rips complex up to dimension 2 and a chosen radius $\delta$, take the vertices of the barycentric subdivision, and sparsify with minimum distance $\epsilon$.
The PH of the alpha complex on the resulting dataset is then computed and compared with the PH of the known source shape.
The results of this experiment are reported in \cref{fig_syntheticcomparison1}.
We then compute the associated grid, thickening, and complement, along with the respective cubical complexes.
One such computation is presented in \cref{fig_syntheticcomparison2}.
Note in \cref{fig_syntheticcomparison2} the visual confirmation that $\mathcal C_\mu(G(X,\mu)$ and $\mathcal C_\mu(\mu\Z^N\setminus G(X,\mu))$ are not complements (as topological spaces) of each other, supporting the choice of spaces in \cref{thm_complement}.

\begin{figure}[htbp]
    \centering
	\includegraphics[width=15cm]{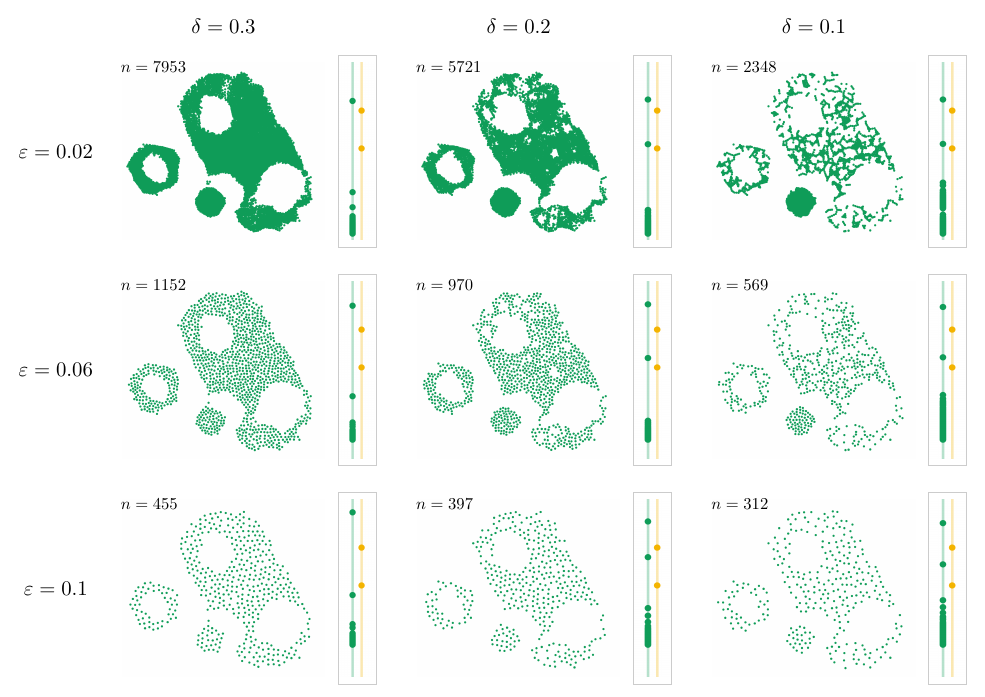}
    \caption{The output after applying \texttt{barycentric\_subdivison}, followed by \texttt{sparsification}, to the input dataset, for $\delta\in\{0.1,0.2,0.3\}$ and $\epsilon\in\{0.02,0.06,0.1\}$. Persistence diagrams in degree 0 and sizes of the outputs are also displayed.}
	\label{fig_syntheticcomparison1}
\end{figure} 

Note that for large values of $\delta$ the persistence pairs which may be considered noise (heuristically, the points that cluster near 0 in the diagrams of \cref{fig_syntheticcomparison1}) have a lower spread than for small values of $\delta$.
This may be considered as a type of \emphasize{noise reduction}, as all the empty regions bounded by data points within a distance $\delta$ of each other now get a new data point in the middle, at most halving the length of a bar in the barcode.


\begin{figure}[htbp]
    \centering
	\includegraphics[width=15cm]{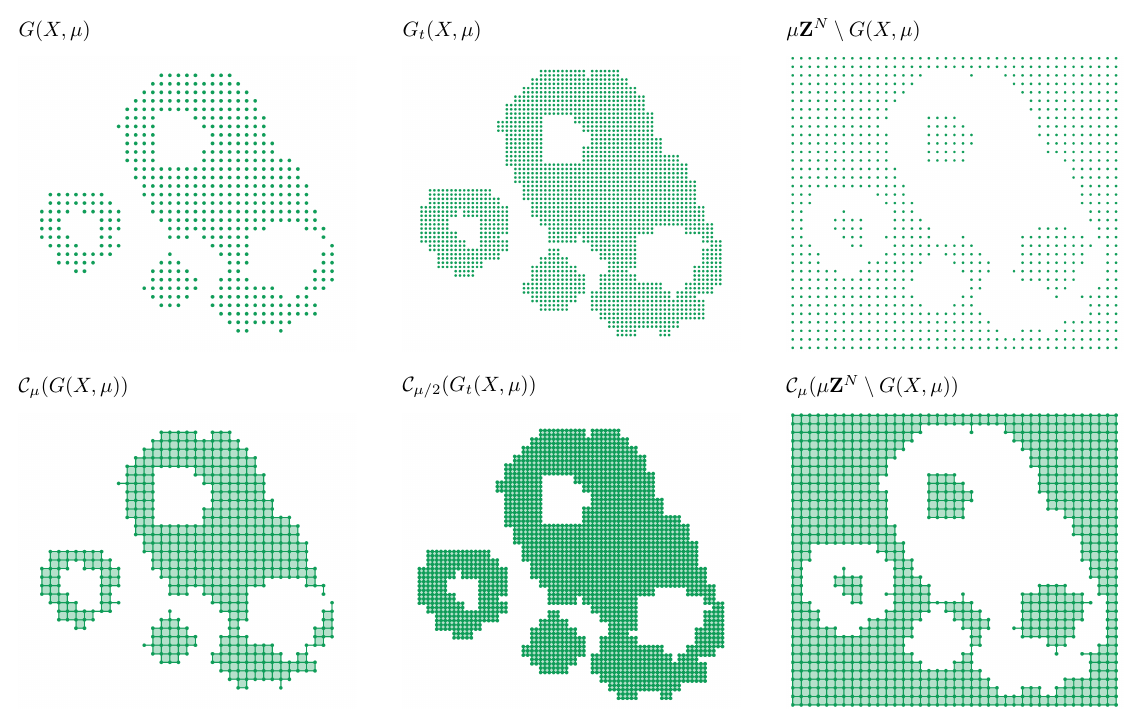}
	\caption{Grids (top row) and cubical complexes (bottom row) computed on the output of one of the point clouds from \cref{fig_syntheticcomparison1}. The point cloud used is $S(B(X,0.3),0.06)$, with $\mu=0.12$.}
	\label{fig_syntheticcomparison2}
\end{figure}

\subsection{Implementation}
The implementation follows a user-motivated approach, allowing for input and output in standard lists or NumPy arrays.
The Vietoris--Rips complex, alpha complex,  complex followed by operations on its simplices, and sparsification, executed as a sequential process through the underlying set, keeping or dropping each element based on distances computed so far (see \cref{sec_structures} for details).
The map $X\to G(X,\mu)$ is also implemented, for $\mu,z$ selected by the user.
The simplicial complex constructions, persistence computations, and bottleneck distance measurements are all done by GUDHI \cite{gudhi:SubSampling}, though other similar PH software may also be used for the same tasks.
\begin{lstlisting}[language=C++]
// C++
./barycentric_subdivision --radius 0.3 --max_dim 2 \
 --data_in  "../examples/generated_data.csv" --data_out "bs.csv" 
./sparsification --min_dist 0.06 \
 --data_in  "bs.csv" --data_out "sp.csv"
./gridification --grid_interval .12 \
 --data_in  "sp.csv" --data_out "gr.csv"
./complement --grid_interval .12 --buffer 2 \
 --data_in  "gr.csv" --data_out "co.csv"
./thickening --grid_interval .12 \
 --data_in  "gr.csv" --data_out "th.csv"

# Python
exec(open("topoaware.py").read())
data = np.load('../examples/generated_data.npy')
bs = barycentric_subdivision(data, radius=0.3, max_dim=2)
sp = sparsification(bs, min_dist=0.06)
gr = gridification(sp, grid_interval=0.12, grid_origin=[0,0])
co = complement(gr, grid_interval=0.12, buffer=2)
th = thickening(gr, grid_interval=0.12)

# R
source("topoaware.r")
data <- read.csv("../examples/generated_data.csv")
bs <- barycentric_subdivision(data, radius=0.3, max_dim=2)
sp <- sparsification(bs, min_dist=0.06)
gr <- gridification(sp, grid_interval=0.12, grid_origin=c(0,0))
co <- complement(gr, grid_interval=0.12, buffer=2)
th <- thickening(gr, grid_interval=0.12)
\end{lstlisting}
GUDHI is written in C++, the functions then being wrapped to work in Python, and further the Python functions wrapped to use in R. 
The R language is most often used in computational biology, and as GUDHI has been ported to R \cite{rgudhi}, we are able to provide self-contained code in R as well.
This particular implementation in R does not contain all the functions that the C++ and Python variants do, but for our purposes there are enough functions.

\subsection{Comparison with a statistical method}
We briefly present a comparison of the same synthetic data as in \cref{sec_syntheticdata}, when used as input for the methods of the R package \texttt{hypervolume} \cite{blonder2018}.
Other methods exist (for example, \cite{elith2006, qiao2016nichea}), but as the goal of this section is more to build a bridge to an applied area, rather than to evaluate the benefits of different approaches, we stick to a closer look at only one approach.
The primary target audience for the \texttt{hypervolume} package is ecologists, and the outputs of the methods of \nameshort~are similar to those of \texttt{hypervolume}, though \texttt{hypervolume} provides a a much wider array of tools with which to infer properties of and perform operations on point clouds (for example volume, union).

\begin{figure}[htbp]
    \centering
	\includegraphics[width=\textwidth]{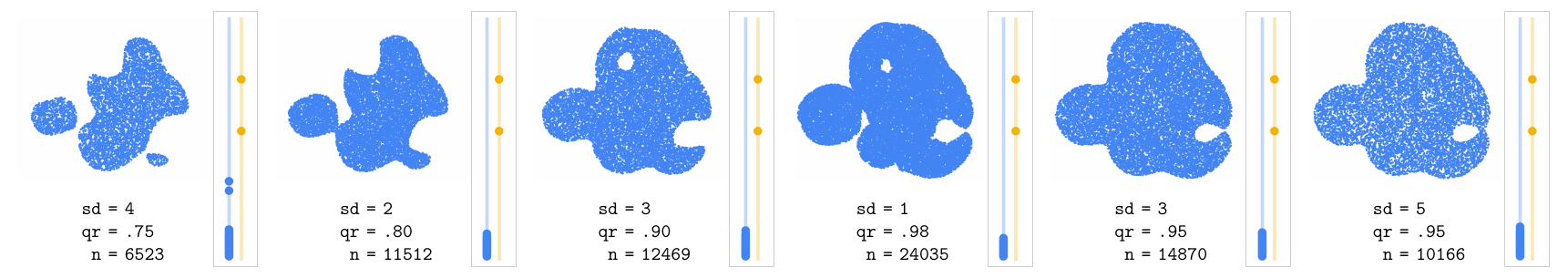}
	\caption{The source sample from \cref{fig_syntheticdata} input into the statistical methods of the R package \texttt{hypervolume} \cite{blonder2018}. The standard deviation of the distribution (\texttt{sd}) and the certainty quantile requested (\texttt{qr}) are indicated, along with the number of points that these parameters yield (\texttt{n}). The default values are \texttt{sd=3} and \texttt{qr=.95} (second right). Alpha complex persistence diagrams $D_0$ are compared with that of the source shape.}
	\label{fig_blonder}
\end{figure} 

A visual overview of outputs for different parameters is presented in \cref{fig_blonder}.
Note that the general shape is somewhat comparable, though the 0-dimensional topological features in most cases have been lost (compare with the $D_0$ diagrams in \cref{fig_syntheticcomparison1}).
This statistical approach may be considered complementary to the barycentric subdivision approach: \texttt{hypervolume} infers new samples \emphasize{around} each existing sample, while \nameshort~infers new sample \emphasize{between} collections of existing samples.

\section{Discussion}
\label{sec_discussion}

In this work we have presented tools with which to modify finite metric spaces, demonstrating that topological information is still preserved, within bounds related to the parameters of modification.
We focus on dimension 0 and codimension 1 homology, for interpretability and ease of computation, in particular for applications from other areas of science, such as computational ecology. 

As the setup only requires an input point cloud with a notion of distance, our methods are well-suited to any dataset, in particular large ones for which computations are resource-intensive.
As is usually the case with new inferences being made from existing observations, we must be careful to consider the new data points as mathematical constructions, not necessarily having precise preimages in the real world setting from which the input data was taken.

Improvements of \nameshort~include adjusting to the interests of the applied computational community, for example computing the $(N-1)$-volume of the boundary and the $N$-volume of the alpha complex, as this is something that ecologists are interested in and that comparable methods provide \cite{elith2006,blonder2018}.
Considering computational improvements, we note that barycentric subdivision and (less so) sparsification lend themselves to parallelization, though the order of elements and the rejoining of outputs after a parallel execution requires considerations not discussed here.
As a next step for developing these ideas, we intend to present these methods to an ecological audience, and provide more examples, comparisons, and interpret the results in the context of that field.


\paragraph{Data availability.} 
The implementation of the methods described here is available on a public GitHub repository \cite{Lazovskis_TopoAware_2024}. 
The synthetic data generated and processed in \cref{sec_syntheticdata}, as well as the particular functions applied to generate data for visualizations, are also available at the same location.

\paragraph{Contributions.} 
All authors contributed equally to the initial vision of this paper. JL wrote the first draft and created the visualizations. All authors contributed equally to editing and preparing the final draft.

\paragraph{Acknowledgements.} 
JL, RL, JM were supported by the University of Cambridge and the Issac Newton Institute's Retreat Programme in October 2024, during which the initial vision of this paper was developed.
JL was partially supported by the Latvian Council of Science (LZP) 1.1.1.9 Research application No 1.1.1.9/LZP/1/24/125 of the Activity ``Post-doctoral Research" ``Efficient topological signatures for representation learning in medical imaging".

\printbibliography

\end{document}